\documentclass[11pt,leqno]{amsart}

\usepackage[colorlinks=false,hidelinks]{hyperref} 
\usepackage[T1,T5]{fontenc} 
\usepackage[vietnamese,english]{babel}
\usepackage{mathtools}
\usepackage{amssymb}
\usepackage[shortlabels]{enumitem}
\usepackage{graphicx}
\usepackage{orcidlink}
\usepackage{quiver}
\usetikzlibrary{babel}
\usepackage[initials]{amsrefs}

\newtheorem{thm}{Theorem}[section]
\newtheorem{lemma}[thm]{Lemma}
\newtheorem{cor}[thm]{Corollary}
\newtheorem{prop}[thm]{Proposition}

\theoremstyle{definition}
\newtheorem{definition}[thm]{Definition}
\newtheorem{example}[thm]{Example}
\newtheorem{ex}[thm]{Example}
\newtheorem*{ack}{Acknowledgments}

\theoremstyle{remark}
\newtheorem{rmk}[thm]{Remark}
\newtheorem{com}[thm]{Comment}
\newtheorem*{notation}{Notation}

\numberwithin{equation}{section}

\newcommand{\id}{\mathrm{id}}
\newcommand{\inv}{^{-1}}
\newcommand{\R}{\mathbb{R}}
\renewcommand{\phi}{\varphi}
\newcommand{\eps}{\varepsilon}
\newcommand{\bij}{\xrightarrow{\sim}}
\newcommand{\Grp}{\mathsf{Grp}}
\newcommand{\Set}{\mathsf{Set}}
\newcommand{\Rack}{{\mathsf{Rack}_\ast}}
\newcommand{\Corack}{{\mathsf{Corack}_\ast}}
\newcommand{\Alg}{{\mathsf{CAlg}_k}}
\newcommand{\Algs}{{\mathsf{CAlg}_S}}
\newcommand{\Sch}{{\mathsf{Sch}_k}}
\newcommand{\Schs}{{\mathsf{Sch}_S}}
\newcommand{\Aff}{{\mathsf{Aff}_k}}
\newcommand{\Affs}{{\mathsf{Aff}_S}}
\newcommand{\leib}{{\mathsf{Leib}_k}}
\newcommand{\Rleib}{{\mathsf{RLeib}_k}}
\newcommand{\op}{^{\mathrm{op}}}
\newcommand{\tr}{\triangleright}

\DeclareMathOperator{\OL}{OL}
\newcommand{\ol}{\mathfrak{ol}}
\DeclareMathOperator{\GL}{GL}
\DeclareMathOperator{\Conj}{Conj}
\renewcommand{\O}{\mathcal{O}}
\DeclareMathOperator{\Der}{Der}
\DeclareMathOperator{\Ad}{Ad}
\DeclareMathOperator{\ad}{ad}
\DeclareMathOperator{\Spec}{Spec}
\newcommand{\C}{\mathcal{C}}
\newcommand{\D}{\mathcal{D}}
\newcommand{\fo}{f_{(1)}}
\newcommand{\ft}{f_{(2)}}
\newcommand{\go}{g_{(1)}}
\newcommand{\gt}{g_{(2)}}
\newcommand{\uo}{u_{(1)}}
\newcommand{\ut}{u_{(2)}}
\newcommand{\vo}{v_{(1)}}
\newcommand{\vt}{v_{(2)}}
\renewcommand{\q}{\mathfrak{q}}
\newcommand{\YB}{\mathsf{YB}}
\newcommand{\nd}{\mathsf{YB}^{\mathrm{nd}}}
\newcommand{\Racks}{\mathsf{Rack}}

\hypersetup{
	pdflang={en-US},
	pdftitle={From affine algebraic racks to Leibniz algebras and Yang--Baxter operators},
	pdfauthor={Lực Ta},
	pdfsubject={Mathematics},
	pdfkeywords={Affine algebraic rack, algebraic group, braid equation, conjugation quandle, Leibniz algebra, nondegenerate Yang--Baxter solution, pointed rack object}
}

\begin{document}
	
	\title[
    From affine algebraic racks to Leibniz algebras
    ]{
    From affine algebraic racks to Leibniz algebras and Yang--Baxter operators
    }
	\author{L\d\uhorn c Ta \orcidlink{0009-0008-6824-186X}}	
	
	\address{Department of Mathematics, University of Pittsburgh, Pittsburgh, Pennsylvania 15260}
	\email{ldt37@pitt.edu}
	
	
	
	\begin{abstract}
We introduce analogues of algebraic groups called \emph{algebraic racks}, which are pointed rack objects in the category of schemes over a ground field. 
Addressing a problem of Loday \cite{loday}, we construct functors assigning left and right Leibniz algebras to affine algebraic racks. These functors are compatible with closed subracks and ideals, and they recover the Lie algebras of linear algebraic groups (via conjugation quandles) and the Leibniz algebras of algebraic Lie racks. We also study properties of coordinate algebras and Leibniz algebras of affine algebraic racks. 
Finally, we use rack schemes to functorially construct (co-)nondegenerate Yang--Baxter operators in various categories. 
	\end{abstract}
	\maketitle
	
\section{Introduction}
This paper introduces analogues of algebraic groups called \emph{rack schemes} (resp.\ \emph{algebraic racks}), which are racks (resp.\ pointed racks) internal to the category of schemes of finite type over a ground field $k$. We show that affine algebraic racks provide a solution to a version of the well-known \emph{coquecigrue} problem for left and right Leibniz algebras over arbitrary fields $k$ (see Theorem \ref{thm:main}). We also use rack schemes to construct solutions to the Yang--Baxter equation in the categories of schemes, commutative $k$-algebras, and sets (see Sections \ref{sec:from3} and \ref{subsec:yb-ol}).

\subsubsection{}
Loday's \cite{loday} original version of the \emph{coquecigrue} problem from 1993 seeks algebraic structures on smooth manifolds that differentiate to Leibniz algebras in a way that recovers the Lie group--Lie algebra correspondence. 
Bloh \cite{bloh} introduced Leibniz algebras in 1965 to generalize Lie algebras.

To address the \emph{coquecigrue} problem over $\R$ and $\mathbb{C}$, Kinyon \cite{kinyon} in 2007 introduced \emph{Lie racks}, which are pointed rack objects (see Section \ref{sec:ptd-racks}) in the category of smooth manifolds; see the survey \cite{ongay} for further discussion. \emph{Racks} \cite{fenn} and related structures called \emph{quandles} \citelist{\cite{takasaki}\cite{joyce}\cite{matveev}} were originally introduced as invariants of knots and 3-manifolds; see \citelist{\cite{book}\cite{quandlebook}} for introductions to the theory. 

\subsection{Main results}
Kinyon \cite{kinyon} posed an analogue of the \emph{coquecigrue} problem over arbitrary fields $k$ that replaces Lie groups with algebraic groups. Save for a preliminary report of Ba\v{s}i\'{c} and \v{S}koda \cite{skoda} from 2008, the author is unaware of any work in this direction. 
To address this, we prove the following theorem. Working over a ground field $k$, let $\Rack(\Aff)$, $\leib$, and $\Rleib$ denote the categories of affine algebraic racks, left Leibniz algebras, and right Leibniz algebras, respectively.

\begin{thm}\label{thm:main}
    Let $Q=\Spec(A)$ be an affine algebraic rack with unit $e\in Q$, and let $\q\coloneq \Der_k(A,k)$ be the vector space of $k$-derivations of $A$. Define two Leibniz brackets $[\cdot,\cdot],\{\cdot,\cdot\}\colon\q^2\to\q$ by convolution with respect to the corack operations $(\nabla,\nabla\inv)\coloneq(\tr^\sharp,(\tr\inv)^\sharp)$ of $A$ (see Section \ref{sec:nabla}):
    \[
    [D,E]\coloneq (D\otimes E)\circ\nabla,\qquad \{D,E\}\coloneq (D\otimes E)\circ\nabla\inv.
    \]
    Then the assignments $Q\mapsto(\q,[\cdot,\cdot])$ (resp.\ $Q\mapsto(\q,\{\cdot,\cdot\})$) and $\phi\mapsto (d\phi)_e$ define a covariant functor $\Rack(\Aff)\to\leib$ (resp.\ $\Rack(\Aff)\to\Rleib$).
    
    Furthermore, if $Q=\Conj(G)$ is the conjugation quandle scheme of an affine algebraic group $G$, then $(\q,[\cdot,\cdot])$ is the Lie algebra of $G$. On the other hand, if $k=\R$ or $\mathbb{C}$ and $Q(k)$ is a Lie rack, then $(\q,[\cdot,\cdot])$ is the Leibniz algebra of $Q(k)$ as a Lie rack.
\end{thm}

\begin{proof}
    For $[\cdot,\cdot]$, we prove the theorem piecemeal in Propositions \ref{prop:adj}, \ref{prop:jacobi}, and \ref{prop:functor} and Remark \ref{rmk:recover}. 
    The proof for $\{\cdot,\cdot\}$ is similar.
\end{proof}

\subsubsection{}
Theorem \ref{thm:main} is an algebraic analogue of a result of Kinyon's result for Lie racks in \cite{kinyon}*{Thm.\ 3.4}. 
Comparatively, two advantages of Theorem \ref{thm:main} are that it holds for all fields $k$, equips $\q$ with both left and right Leibniz brackets, and gives explicit formulas for both brackets. Like in Kinyon's construction, the left Leibniz bracket $[\cdot,\cdot]$ agrees with the derivative of the \emph{adjoint representation} of $Q$ at the identity in the sense of Proposition \ref{prop:adj}. 

The last part of Theorem \ref{thm:main} states that the usual functors sending affine algebraic groups to their Lie algebras (see \cite{waterhouse}*{Sec.\ 12.2}) and affine algebraic Lie racks over $\R$ or $\mathbb{C}$ to their (left) Leibniz algebras (see \cite{kinyon}) factor through the functor $\Rack(\Aff)\to\leib$. As racks and quandles have enjoyed increasing interest in quantum algebra and algebraic geometry in the past decade, this paper lays the foundations for what we expect will be a very rich area of study; cf.\ Section \ref{sec:open}.

\subsubsection{}
As an application of rack schemes over commutative rings $S$, we also construct solutions to the Yang--Baxter equations in various categories; see Sections \ref{sec:from3} and \ref{subsec:yb-ol}. Our constructions employ a fully faithful functor from the category of rack schemes to the category of \emph{Yang--Baxter schemes} (that is, braided objects in the category of schemes over $S$); see Theorem \ref{thm:yb}. 

This functor's action on objects originally appeared in \cite{crans} and is a more user-friendly specialization of a construction from \citelist{\cite{carter}\cite{lebed}\cite{guccione}}. See Comment \ref{com:yb-motivations} for further discussion and references.

\subsection{Literature review}
It appears that the only existing work on the algebraic \emph{coquecigrue} problem over arbitrary fields $k$ is a preliminary report of unpublished work of Ba\v{s}i\'{c} and \v{S}koda \cite{skoda} from 2008 in the case that $\operatorname{char}(k)=0$. On the other hand, there is a rich literature on the \emph{coquecigrue} problem for smooth manifolds over $\R$ and $\mathbb{C}$. See \cite{ongay} for a discussion of the literature up to 2014.

Although the first work on the \emph{coquecigrue} problem is due to Datuashvili \cite{Datuashvili} in 2004, most modern approaches to the problem employ Lie racks, which Kinyon \cite{kinyon} introduced in 2007 in response to the problem. 
For example, Covez \cite{covez} in 2013 used \emph{local augmented Lie racks} to provide an analogue of Lie's third theorem for Leibniz algebras over $\R$ with the caveat that a rack structure is only defined in a neighborhood of the identity. By contrast, Bordemann and Wagemann \cite{bordemann} in 2017 introduced a non-functorial way to integrate Leibniz algebras over $\R$ into Lie racks. 

Within the last decade, various authors have classified special types of Lie racks \citelist{\cite{abchir2}\cite{benayadi}\cite{dibartolo}}. These include Lie racks corresponding to certain nilpotent Leibniz algebras \citelist{\cite{larosa}\cite{larosa2}} and symmetric Leibniz algebras \cite{abchir}, as well as certain Lie racks that integrate to Lie algebras \cite{larosa}*{Thm.\ 4.6}. For further work on the \emph{coquecigrue} problem, see \citelist{\cite{alexandre}\cite{dherin}}. For non-rack-theoretic approaches, see \citelist{\cite{smith}\cite{rodriguez}\cite{uslu}}.

\subsubsection{}
Although there does not appear to be existing literature on rack schemes, we note that quandle varieties and abstract racks have enjoyed increasing attention in algebraic geometry in the past decade. In particular, quandle varieties provide algebro-geometric analogues of Riemannian symmetric spaces \citelist{\cite{takahashi}\cite{takahashi2}} with applications to knot theory \cite{agl} and geometric invariant theory \cite{fan}. 

Abstract racks and quandles provide invariants of not only knots and 3-manifolds but also various objects of study in arithmetic algebraic geometry. For example, quandles provide invariants of arithmetic schemes \cite{takahashi1} and categorical constructions of Hurwitz spaces \citelist{\cite{bianchi}\cite{bianchi2}}. Racks have similar applications in the theory of Hurwitz spaces and algebraic function fields \citelist{\cite{randal}\cite{shusterman}\cite{landesman}\cite{ellenberg}}. The increasing attention racks and quandles have received in algebraic geometry provides another motivation for developing a modern scheme-theoretic treatment of algebraic racks. 

\subsubsection{}
Rack schemes functorially induce rich families of solutions to the Yang--Baxter equations. 
The first algebro-geometric approaches to Yang--Baxter operators are due to Etingof \cite{etingof2} in 2003, who studied nondegenerate braided sets induced by complex Yang--Baxter varieties with birational morphisms. These were later studied in \citelist{\cite{adler}\cite{suris}\cite{frieden}}. Further connections between algebraic geometry and the Yang--Baxter equations  were studied in \citelist{\cite{gorbunov}\cite{inoue}}. 
After this paper was posted to arXiv, Ma, Zhang, and Liu \cite{braces} gave examples of affine Yang--Baxter schemes induced by affine skew braces.

\subsection{Notation and conventions}
    Let $k$ be a field, and let $S$ be a commutative ring.
    By a \emph{scheme}, we mean a scheme of finite type over $k$ or $S$. Unless we are specifically discussing Leibniz algebras, by \emph{$k$-algebras} we mean finitely generated commutative associative algebras over $k$, and similarly for \emph{$S$-algebras}. Denote the $k$-algebra of \emph{dual numbers} by $k[\delta]\coloneq k[\delta]/(\delta^2)$. Given a morphism of affine schemes $\phi\colon\Spec(B)\to\Spec(A)$, let $\phi^\sharp\colon A\to B$ denote the induced ring homomorphism. 
    
    Let $\Schs$, $\Affs$, $\Algs$, $\leib$, $\Rleib$, $\Grp$, $\Grp(\Sch)$, $\Set$, $\Racks$, and $\Rack$ denote the categories of schemes over $S$, affine schemes over $S$, $S$-algebras, left Leibniz algebras over $k$, right Leibniz algebras over $k$, groups, algebraic groups over $k$, sets, racks, and pointed racks, respectively. 
    
    Given an object $X$ in a monoidal category $(\C,\otimes)$, let $X^n\coloneq \bigotimes^n_{i=1}X$ denote the $n$-fold tensor product of $X$. If $(\C,\otimes)$ is braided monoidal, let $\tau$ denote the braiding $\tau\colon X^2\bij X^2$. If $(\C,\times)$ is cartesian monoidal, then let $\pi_1,\pi_2\colon X^2\to X$ denote the projections (so, for example, $\tau=(\pi_2,\pi_1)$).

\subsection{Organization of the paper} The first few sections of this paper are dedicated to preliminaries.
In Section \ref{sec:prelims}, we recall the definitions of Leibniz algebras and pointed racks and give several important examples.

In Section \ref{sec:ptd-racks}, we introduce pointed rack objects internal to cartesian monoidal categories.

In Section \ref{sec:alg-racks}, we introduce the categories of algebraic racks $\Rack(\Sch)$ and affine algebraic racks $\Rack(\Aff)$. We also discuss the functor of points approach for algebraic racks.

In Section \ref{sec:corack}, we introduce the category of (commutative) corack algebras $\Corack(\Alg)$.

\subsubsection{} 
Looking toward Theorem \ref{thm:main}, we specialize our attention to the category of affine algebraic racks $\Rack(\Aff)$ over a ground field $k$.
In Section \ref{sec:affine}, we discuss the anti-equivalence between the categories $\Rack(\Aff)$ and $\Corack(\Alg)$. As examples, we compute the corack algebras corresponding to three classes of affine algebraic racks. We also show that, unlike Hopf algebras of algebraic groups, corack algebras are coassociative or cocommutative only in trivial cases.

In Section \ref{sec:pre-pf}, we discuss our approach to proving Theorem \ref{thm:main}. In particular, we define adjoint representations $\Ad$, $\ad$ and dual adjoint representations $\Ad\inv,$ $\ad\inv$ of affine algebraic racks and their tangent spaces. We also recall various facts about $k$-derivations and tangent spaces.

In Section \ref{sec:pf}, we prove Theorem \ref{thm:main}. A major step in the proof is to show that, in a certain sense, the Leibniz brackets $[\cdot,\cdot]$ and $\{\cdot,\cdot\}$ respectively agree with the maps $\ad$ and $\ad\inv$.

In Section \ref{sec:from}, we provide several results relating the structure of an affine algebraic rack to that of its Leibniz algebra, including sufficient conditions for the latter to be abelian. These results recover the analogous results for Lie algebras of affine algebraic groups.

\subsubsection{}
Next, we discuss how rack schemes $\Racks(\Schs)$ over a commutative ring $S$ induce solutions to the Yang--Baxter equation in $(\Schs,\times_S)$, $(\Set,\times)$, and $(\Algs,\otimes_S)$.
In Section \ref{sec:ybo}, we recall definitions relating to Yang--Baxter operators $\YB(\C)$ internal to a tensor category $(\C,\otimes)$.

In Section \ref{sec:from2}, we show that for every cartesian monoidal category $(\C,\times)$, a certain construction in \citelist{\cite{lebed}\cite{guccione}} defines a fully faithful functor $F\colon \Racks(\C)\to\YB(\C)$. (See Theorem \ref{thm:yb}.)

In Section \ref{sec:from3}, we apply the results of Sections \ref{sec:ybo} and \ref{sec:from2} to $\Schs$ and $\Affs$. In particular, we discuss how rack schemes induce nondegenerate Yang--Baxter operators in $\Schs$ and $\Set$ and co-nondegenerate Yang--Baxter operators in $\Algs$.

\subsubsection{}
In Section \ref{sec:ol}, we illustrate Theorem \ref{thm:main} with a class of affine algebraic racks $\OL_n$ whose Leibniz algebras $\ol_n(k)$ are not generally Lie algebras. To illustrate the constructions in Section \ref{sec:from3}, we also compute the induced Yang--Baxter operators in $\Aff$ and $\Alg$.

In Section \ref{sec:open}, we propose directions for future research.

\begin{ack}
    An early draft of this paper was submitted as the final project for a course on affine algebraic groups taught by Prof.\ Bogdan Ion at the University of Pittsburgh in the fall 2025 semester.
    I thank Prof.\ Ion for helpful comments and suggestions. I also thank Gianmarco La Rosa and Manuel Mancini for helpful clarifications about Lie racks.
    
    This work is supported by the K.\ Leroy Irvis Fellowship at the University of Pittsburgh.
\end{ack}

\section{Preliminaries}\label{sec:prelims}
In this section, we recall the definitions of Leibniz algebras and pointed racks. 
See \cite{leibniz} for an introduction to Leibniz algebras, and see \citelist{\cite{book}\cite{quandlebook}} for general references on racks and quandles. 

\subsection{Leibniz algebras}
We recall the definitions of left and right Leibniz algebras over $k$.

Let $\q$ be a vector space over $k$, and let $[\cdot,\cdot],\{\cdot,\cdot\}\colon \q^2\to\q$ be bilinear mappings. We call $[\cdot,\cdot]$ a \emph{(left) Leibniz bracket} and $(\q,[\cdot,\cdot])$ a \emph{left Leibniz algebra} if, for all $X\in\q$, the map $\ad_X\colon\q\to\q$ defined by $Y\mapsto[X,Y]$ is a derivation of $[\cdot,\cdot]$. That is, the left derivation form of the Jacobi identity
    \[
    [X,[Y,Z]]=[[X,Y],Z]+[Y,[X,Z]]
    \]
holds for all $X,Y,Z\in\q$. 

Similarly, we call $\{\cdot,\cdot\}$ a \emph{(right) Leibniz bracket} and $(\q,\{\cdot,\cdot\})$ a \emph{right Leibniz algebra} if for all $X\in\q$, the map $\ad_X\inv\colon\q\to\q$ defined by $Y\mapsto\{Y,X\}$ is a derivation of~$\{\cdot,\cdot\}$:
\[
\{\{X,Y\},Z\}=\{\{X,Z\},Y\}+\{X,\{Y,Z\}\}.
\]

If $(\q,[\cdot,\cdot]_\q)$ and $(\mathfrak{r},[\cdot,\cdot]_\mathfrak{r})$ are left Leibniz algebras, then a $k$-linear map $\phi\colon \q\to\mathfrak{r}$ is called a \emph{(left) Leibniz algebra homomorphism} if it preserves the respective Leibniz brackets. That is,
\[
\phi([X,Y]_\q)=[\phi(X),\phi(Y)]_\mathfrak{r}
\]
for all $X,Y\in\q$. \emph{Right Leibniz algebra homomorphisms} are defined identically.

\subsubsection{}
For example, a Lie algebra is precisely a left or right Leibniz algebra whose Leibniz bracket is alternating; that is, $[X,X]=0$ for all $X\in\q$.

There are many examples of left Leibniz algebras that are not Lie algebras. Of particular interest in this paper is the following construction of Weinstein \cite{weinstein} in 2000; see \cite{omni}. 
    Given a vector space $V$, consider the vector space $\mathfrak{gl}(V)\oplus V$. (This is the underlying vector space of the Lie algebra $\mathfrak{aff}(V)$ of the linear algebraic group $\mathrm{Aff}(V)$ of affine transformations of $V$.) 
    
    Define a left Leibniz bracket on $\mathfrak{gl}(V)\oplus V$ by
    \begin{equation}\label{eq:olnk}
        [(X,v),(Y,w)]\coloneq (XY-YX,Xw).
    \end{equation}
    Then $\ol(V)\coloneq(\mathfrak{gl}(V)\oplus V,[\cdot,\cdot])$ is a hemisemidirect product left Leibniz algebra (see \cite{kinyon2}) called an \emph{omni-Lie algebra}. Despite its name, $\ol(V)$ is a Lie algebra if and only if $\dim (V)=0$. (See \cite{weinstein} for the original motivation for this name; cf.\ \cite{omni}*{Prop.\ 3.1}.) If $V=k^n$, then we denote $\ol_n(k)\coloneq\ol(k^n)$.

\subsection{Pointed racks}

To motivate the definition of a pointed rack object, we begin in the set-theoretic setting of ordinary rack and quandle theory. Recall (from, say, \cite{kinyon}) that a \emph{pointed rack} consists of a set $Q$, a distinguished element $e\in Q$ called the \emph{unit} or \emph{identity}, and two binary operations $\tr,\tr\inv\colon Q^2\to Q$ satisfying the following five axioms\footnote{Some authors swap the roles of $\tr$ and $\tr\inv$ in the definition of a rack. Of course, this choice is purely notational. This paper adopts the convention that $\tr$ is left-distributive because this seems to be the more common convention in the literature on Lie racks and quandle varieties.} for all elements $x,y,z\in Q$:
    \begin{enumerate}[(Q1)]
        \item (Invertability) $(x\tr y)\tr\inv x=y=x\tr(y\tr\inv x)$.
        \item\label{ax:left} (Left-distributivity) $x\tr(y\tr z)=(x\tr y)\tr (x\tr z)$.
        \item (Right-distributivity) $(x\tr\inv y)\tr\inv z=(x\tr\inv z)\tr\inv(y\tr\inv z)$.
        \item\label{ax:fixing} (Fixing) $e\tr x=x$.
        \item\label{ax:fixed} (Fixedness) $x\tr e=e$.
    \end{enumerate}
If $x\tr(x\tr y)=y$ (equivalently, $x\tr y= y\tr\inv x$) for all $x,y\in Q$, then we say that $Q$ is \emph{involutory}. If $x\tr x=x$ for all $x\in Q$, then we say that $Q$ is \emph{idempotent} or a \emph{(pointed) quandle}.

Given pointed racks $Q$ and $R$, a function $\phi\colon Q\to R$ is called a \emph{pointed rack homomorphism} if
    \[
    \phi(x\tr_Q y)=\phi(x)\tr_R\phi(y),\qquad \phi(e_Q)=e_R
    \]
    for all $x,y\in Q$. (In this case, $\phi$ also preserves the right-distributive operations $\tr_Q\inv$ and $\tr_R\inv$. Verifying this fact is easiest when using the notation $L_x^{\pm 1}$ in Section \ref{subsec:lx}.) We can similarly define \emph{racks} and \emph{rack homomorphisms} by omitting every axiom that involves the unit $e$; cf.\ \cite{fenn}.

Note that every pointed rack $Q$ also satisfies the following identities for all $x\in Q$:
\begin{enumerate}[(Q1')]
\setcounter{enumi}{3}
        \item\label{ax:inv-fixing} (Inverse fixing) $x\tr\inv e=x$.
        \item\label{ax:inv-fixed} (Inverse fixedness) $e\tr\inv x=e$.
    \end{enumerate}

\subsubsection{}
For example, every nonempty set $Q$ and choice of unit $e\in Q$ can be made into a \emph{trivial quandle} by defining $x\tr y\coloneq y\eqcolon y\tr\inv x$ for all $x,y\in Q$. A nontrivial example of a pointed quandle, and indeed the original motivating example in \citelist{\cite{joyce}\cite{matveev}}, is the \emph{conjugation quandle} $\Conj(G)$ of a group $G$. The unit of $\Conj(G)$ is the identity of $G$, and the binary operations of $\Conj(G)$ are defined by
\[
g\tr h\coloneq  ghg\inv,\qquad h\tr\inv g\coloneq g\inv hg
\]
for all elements $g,h\in G$. Note that $\Conj(G)$ is involutory if and only if $g^2\in Z(G)$ for all $g\in G$.

\subsubsection{}\label{subsec:lx}
Before we give an example of a pointed rack that is not a quandle, we note that in order to construct a pointed rack $Q$ with identity element $e$, it is enough to define a binary operation $\tr\colon Q^2\to Q$ satisfying axioms \ref{ax:left}, \ref{ax:fixing}, and \ref{ax:fixed} such that for all $x\in Q$, the function $L_x\colon Q\to Q$ defined by $y\mapsto x\tr y$ is bijective. This data completely determines the right-distributive binary operation $\tr\inv\colon Q^2\to Q$. 

In fact, for all elements $x\in Q$, the function $L_x$ is a rack automorphism of $Q$ called \emph{left multiplication by $x$}. The inverse function $L_x\inv$ is given by $y\mapsto y\tr\inv x$. Since $L_x$ is a conjugation map in the case that $Q=\Conj(G)$ is a conjugation quandle, some authors refer to $L_x$ as an \emph{inner automorphism} of $Q$. See \cite{fenn} for further discussion.

\subsubsection{}
The following example is due to Kinyon \cite{kinyon}, though the notation is our own. Let $G$ be a group with identity element $e_G$, let $V$ be a $G$-module, and let $e\coloneq (e_G,\mathbf{0})\in G\times V$. Define a pointed rack structure on the set $\OL(G,V)\coloneq G\times V$ by
    \[
    (A,v)\tr(B,w)\coloneq (ABA\inv, Aw).
    \]
By the above discussion, $\OL(G,V)$ is a pointed rack called a \emph{linear rack}. Note that although $\OL(G,V)$ is not generally a quandle, $\Conj(G)$ embeds into $\OL(G,V)$ as a subquandle.

In particular, the underlying set of $\OL_n(k)\coloneq\OL(\GL_n(k),k^n)$ is $\mathrm{Aff}_n(k)$, the group of affine transformations of $k^n$. To distinguish $\OL_n(k)$ from the conjugation quandle $\Conj(\mathrm{Aff}_n(k))$, we dub $\OL_n(k)$ an \emph{omni-linear rack}. (See Proposition \ref{prop:leib-of-ol} for the motivation behind this specific name and notation.) Note that $\OL_n(k)$ is a quandle if and only if $n=0$.

\subsubsection{}
Finally, recall from \cite{ryder} that if $Q$ is a pointed rack and $ R\subseteq Q$ is a subset containing $e$, then $R$ is called a \emph{(pointed) subrack} if $L_r(R)=R$ for all $r\in R$ and a \emph{left ideal}\footnote{One may similarly define a \emph{right ideal} of $Q$ by requiring that $r\tr x\in R$ for all $r\in R$ and $x\in Q$. However, as observed in \cite{elhamdadi}*{Prop.\ 4.6}, this condition implies that $R=Q$. Indeed, $e\in R$, so $x=e\tr x \in R$ for all $x\in Q$.} or \emph{normal subrack} of $Q$ if $L_x(R)=R$ for all $x\in Q$.\footnote{Some papers, such as \cite{elhamdadi}, call our notion of a left ideal a ``right ideal.'' We deviate from this terminology to more closely analogize the definition of left ideals in rings.

Moreover, \cite{elhamdadi} and some other papers only require that $L_r(R)\subseteq R$ and $L_x(R)\subseteq R$ in their definitions of subracks and left ideals. These weaker conditions do not always guarantee that $R$ a rack; see \cite{kamada}*{Sec.\ 2} for counterexamples.} In particular, all left ideals are pointed subracks; the converse is not true in general. 

For example, the pointed subquandles of a conjugation quandle $\Conj(G)$ are precisely the subsets $\{e\}\subseteq R\subseteq Q$ that are closed under the action of the subgroup $\langle R\rangle\leq G$ by conjugation, while the left ideals of $\Conj(G)$ are precisely unions of $\{e\}$ with other conjugacy classes of $G$. In particular, all subgroups of $G$ are pointed subquandles of $\Conj(G)$, and all normal subgroups are left ideals.

\section{Pointed rack objects}\label{sec:ptd-racks}
In Section \ref{sec:alg-racks}, we will define algebraic racks as \emph{pointed rack objects} internal to $\Sch$. To that end, this section introduces pointed rack objects in cartesian monoidal categories, following the construction of rack objects in \cite{rack-roll} (cf.\ \citelist{\cite{carter}\cite{crans}\cite{lebed}\cite{guccione}}). 

\subsection{}
Since pointed racks are an algebraic theory in the sense of universal algebra, we can define \emph{pointed rack objects} or \emph{internal pointed racks} in cartesian monoidal categories in the same way that one defines internal groups or internal monoids, for example. 

Our notion of pointed rack objects is based on \emph{unital rack objects}, which Grøsfjeld \cite{rack-roll} introduced in 2021. The difference is that we require a specific choice of unit $e$ and require that morphisms preserve these choices, while unital rack objects do not single out any one candidate for $e$. Also, we use the word ``pointed'' over ``unital'' to align with the convention in the literature on Lie racks.

\subsection{}
In the following, recall that $\pi_1,\pi_2\colon Q^2\to Q$ denote the projection morphisms, and recall that $\tau\colon Q^2\bij Q^2$ denotes the transposition $\tau=(\pi_2,\pi_1)$.

\begin{definition}[Cf.\ \cite{rack-roll}]\label{def:rack-ob}
    Let $\C$ be a cartesian monoidal category with terminal object $*$. A \emph{pointed rack object} or \emph{internal pointed rack} in $\C$ is a quadruple $(Q,e,\tr,\tr\inv)$ (sometimes denoted only by $(Q,\tr)$ or $Q$) where
    \begin{itemize}
        \item $Q$ is an object in $\C$,
        \item $\tr$ and $\tr\inv$ are morphisms $Q^2\to Q$ called the \emph{rack operations} of $Q$, and
        \item $e\colon *\to Q$ is a morphism called the \emph{unit} or \emph{identity} of $Q$
    \end{itemize}
    such that the following five diagrams commute:
\begin{equation}\label{ax:inv}
    \begin{tikzcd}
	{Q^2} & {Q^2} \\
	{Q^2} & {Q^2}
	\arrow["{(\pi_2,\,\tr\inv)}", from=1-1, to=1-2]
	\arrow["{(\tr,\pi_1)}"', from=1-1, to=2-1]
	\arrow["\id", from=1-1, to=2-2]
	\arrow["{(\tr,\pi_1)}", from=1-2, to=2-2]
	\arrow["{(\pi_2,\,\tr\inv)}"', from=2-1, to=2-2]
\end{tikzcd}
\end{equation}

\begin{equation}\label{ax:ld}
    \begin{tikzcd}
	{Q^3} && {Q^3} \\
	{Q^2} & Q & {Q^2}
	\arrow["{(\tr,\,\pi_1)\times\id}", from=1-1, to=1-3]
	\arrow["{\id\times\tr}"', from=1-1, to=2-1]
	\arrow["{\id\times\tr}", from=1-3, to=2-3]
	\arrow["\tr", from=2-1, to=2-2]
	\arrow["\tr"', from=2-3, to=2-2]
\end{tikzcd}
\qquad
\begin{tikzcd}
	{Q^3} && {Q^3} \\
	{Q^2} & Q & {Q^2}
	\arrow["{\id\times(\pi_2,\,\tr\inv)}", from=1-1, to=1-3]
	\arrow["{\tr\inv\times\id}"', from=1-1, to=2-1]
	\arrow["{\tr\inv\times\id}", from=1-3, to=2-3]
	\arrow["{\tr\inv}", from=2-1, to=2-2]
	\arrow["{\tr\inv}"', from=2-3, to=2-2]
\end{tikzcd}
\end{equation}
\begin{equation}\label{ax:unit}
\begin{tikzcd}
	{\ast\times Q} & {Q^2} \\
	& Q
	\arrow["{e\times\id}", from=1-1, to=1-2]
	\arrow["{\pi_2}"', from=1-1, to=2-2]
	\arrow["\tr", from=1-2, to=2-2]
\end{tikzcd}\qquad
\begin{tikzcd}
	{Q\times\ast} & {Q\times Q} \\
	{*} & Q
	\arrow["{\id\times e}", from=1-1, to=1-2]
	\arrow["{\pi_2}"', from=1-1, to=2-1]
	\arrow["\tr", from=1-2, to=2-2]
	\arrow["e", from=2-1, to=2-2]
\end{tikzcd}
\end{equation}
A pointed rack object is called
\begin{itemize}
    \item \emph{involutory} if $\tr\circ \tau=\tr\inv$, and
    \item \emph{idempotent} or a \emph{quandle} if $\tr\circ\Delta=\id=\tr\inv\circ\Delta$, where $\Delta\colon Q\to Q^2$ denotes the diagonal morphism $\Delta=(\id,\id)$.
\end{itemize}
\end{definition}

In analogy with the definition of pointed racks, the diagram in \eqref{ax:inv} may be called the \emph{invertability axiom}, the diagrams in \eqref{ax:ld} may respectively be called the \emph{left-distributivity} and \emph{right-distributivity} axioms, and the diagrams in \eqref{ax:unit} may respectively be called the \emph{fixing} and \emph{fixedness} axioms. 

Note that every pointed rack also satisfies diagrams similar to those in \eqref{ax:unit} that generalize the inverse fixing and inverse fixedness identities \ref{ax:inv-fixing} and \ref{ax:inv-fixed}.

\begin{definition}[Cf.\ \cite{rack-roll}]\label{def:rack-mor}
    Let $Q$ and $R$ be pointed rack objects in $\C$. A morphism $\phi\colon Q\to R$ is called a \emph{morphism of pointed rack objects} if the following two diagrams commute:
\[\begin{tikzcd}
	Q & {Q^2} & Q \\
	R & {R^2} & R
	\arrow["\phi"', from=1-1, to=2-1]
	\arrow["{\tr_Q}"', from=1-2, to=1-1]
	\arrow["{\tr_Q\inv}", from=1-2, to=1-3]
	\arrow["{\phi\times\phi}"', from=1-2, to=2-2]
	\arrow["\phi", from=1-3, to=2-3]
	\arrow["{\tr_R}"', from=2-2, to=2-1]
	\arrow["{\tr_R\inv}", from=2-2, to=2-3]
\end{tikzcd}\qquad
\begin{tikzcd}
	\ast & Q \\
	& R
	\arrow["{e_Q}", from=1-1, to=1-2]
	\arrow["{e_R}"', from=1-1, to=2-2]
	\arrow["\phi", from=1-2, to=2-2]
\end{tikzcd}\]
    Let $\Rack(\C)$ denote the category of pointed rack objects and their morphisms in $\C$.
\end{definition}

\begin{ex}
    Pointed racks are pointed rack objects in $\Set$. Morphisms of pointed racks are just pointed rack homomorphisms. That is, $\Rack=\Rack(\Set)$.
\end{ex}

\begin{ex}[\cite{kinyon}]
    \emph{Lie racks} are pointed rack objects in $\mathsf{Diff}$, the category of smooth manifolds. Morphisms of Lie racks are smooth maps whose underlying set-theoretic functions are pointed rack homomorphisms. That is, the category of Lie racks is $\Rack(\mathsf{Diff})$.
\end{ex}

\begin{ex}[\citelist{\cite{elhamdadi}\cite{rack-roll}}]
    \emph{Pointed topological racks} (called ``topological racks with units'' in \cite{elhamdadi}) are pointed rack objects in $\mathsf{Top}$, the category of topological spaces. Morphisms in $\Rack(\mathsf{Top})$ are continuous maps whose underlying set-theoretic functions are pointed rack homomorphisms.
\end{ex}

\subsection{}
As in \cite{rack-roll}, we can also define \emph{rack objects} and \emph{morphisms of rack objects} in cartesian monoidal categories $\C$ by omitting all axioms relating to units $e$ in Definitions \ref{def:rack-ob} and \ref{def:rack-mor}. Let $\Racks(\C)$ denote the category of rack objects in $\C$.

Similarly to Kinyon \cite{kinyon}, we are mainly interested in pointed rack objects $\Rack(\Aff)$ in the category of affine schemes instead of rack objects $\Racks(\Aff)$. This is because the unit $e$ gives us a canonical choice of tangent space $\q\coloneq T_eQ$ on which to define a Leibniz algebra structure.

\subsection{}
As an important consequence of Definitions \ref{def:rack-ob} and \ref{def:rack-mor}, we obtain automorphisms of rack objects that generalize the inner automorphisms of set-theoretic racks. Namely, let $Q$ be a rack object in $\C$, and identify $*\times Q\cong Q\cong Q\times *$. Given a \emph{global element} $x$ of $Q$ (that is, a morphism $x\colon \ast\to Q$), we have automorphisms $L_x,L_x\inv\colon Q\to Q$ defined by
    \begin{equation}\label{eq:lx}
        L_x\coloneq \tr\circ (x\times \id),\qquad L_x\inv\coloneq \tr\inv\circ(\mathord{\id}\times x).
    \end{equation}
Indeed, $L_x$ and $L_x\inv$ are mutually inverse. The former is called \emph{left multiplication by $x$}; cf.\ \cite{fenn}.

\subsection{}
Finally, we generalize the definitions of subracks and left ideals to arbitrary cartesian monoidal categories.

\begin{definition}\label{def:ideal2}
Let $Q$ and $R$ be pointed rack objects in $\C$. We call $R$ a \emph{subrack} of $Q$ if there exists a monomorphism of pointed rack objects $\iota\colon R\hookrightarrow Q$. We call $R$ a \emph{left ideal} of $Q$ if there also exist morphisms $\widetilde{\tr}\colon Q\times R\to R$ and $\widetilde{\tr}\inv\colon R\times Q\to R$ such that the following two diagrams commute:
\[\begin{tikzcd}
	{R^2} & {Q\times R} & {Q^2} \\
	& R & Q
	\arrow["{\iota\times\id}", hook, from=1-1, to=1-2]
	\arrow["{\tr_R}"', from=1-1, to=2-2]
	\arrow["{\id\times\iota}", hook, from=1-2, to=1-3]
	\arrow["{\widetilde{\tr}}", from=1-2, to=2-2]
	\arrow["{\tr_Q}", from=1-3, to=2-3]
	\arrow["\iota", hook, from=2-2, to=2-3]
\end{tikzcd} \qquad
\begin{tikzcd}
	{R^2} & {R\times Q} & {Q^2} \\
	& R & Q
	\arrow["{\id\times\iota}", hook, from=1-1, to=1-2]
	\arrow["{\tr_R\inv}"', from=1-1, to=2-2]
	\arrow["{\iota\times\id}", hook, from=1-2, to=1-3]
	\arrow["{\widetilde{\tr}\inv}", from=1-2, to=2-2]
	\arrow["{\tr_Q\inv}", from=1-3, to=2-3]
	\arrow["\iota", hook, from=2-2, to=2-3]
\end{tikzcd}\]\end{definition}

\section{Algebraic racks}\label{sec:alg-racks}
In this section, we introduce the categories of algebraic racks $\Rack(\Sch)$ and affine algebraic racks $\Rack(\Aff)$. These are defined similarly to how the categories of algebraic groups $\Grp(\Sch)$ and affine algebraic groups $\Grp(\Aff)$ are defined. 
To that end, we follow the approach taken to define the latter categories in \citelist{\cite{milne}\cite{waterhouse}}; see these and \cite{borel} for general references on algebraic groups.

\subsection{}
Just as algebraic groups are defined to be group objects in $\Sch$, we define algebraic racks to be pointed rack objects in $\Sch$. 
For the purposes of proving Theorem \ref{thm:main}, we are primarily interested in working over a ground field $k$, but the definitions work when $k$ is replaced with an arbitrary commutative ring $S$.

\begin{definition}
    The categories of \emph{algebraic racks} and \emph{affine algebraic racks} are $\Rack(\Sch)$ and $\Rack(\Aff)$, respectively.
\end{definition}

\begin{definition}
    The categories of \emph{rack schemes} and \emph{affine rack schemes} are $\Racks(\Sch)$ and $\Racks(\Aff)$, respectively.
\end{definition}

More explicitly, an algebraic rack consists of a scheme $(Q,\O_Q)$ and three morphisms of schemes $\tr,\tr\inv\colon Q^2\to Q$ and $e\colon \Spec(k)\to Q$ that make the diagrams in Definition \ref{def:rack-ob} commute. 

\begin{definition}
    Let $Q$ be an algebraic rack. An algebraic subrack $R$ of $Q$ (in the sense of Definition \ref{def:ideal2}) is called \emph{closed} if the monomorphism $\iota\colon R\hookrightarrow Q$ is a closed embedding.
\end{definition}

\subsection{}
In the same way that we can consider every (affine) algebraic group as a (representable) functor $ \Alg\to\Set$ that factors through $\Grp$, we can also consider every (affine) algebraic rack as a (representable) functor $\Alg\to \Set$ that factors through $\Rack$. This is the so-called \emph{functor of points approach} that Grothendieck \cite{grothendieck} introduced in 1973; see \cite{milne}*{Sec.\ 1a} or any other modern algebraic geometry text for details.

\begin{ex}[Cf.\ \cite{takahashi}*{Ex.\ 3.2}]\label{ex:conj}
    Let $G$ be an algebraic group, so for all $k$-algebras $R$, the set of $R$-points $G(R)$ inherits a group structure from $G$. The \emph{conjugation quandle scheme} of $G$ is the algebraic quandle $\Conj(G)\colon\Alg\to\Set$ that sends $R$ to the conjugation quandle $\Conj(G(R))$. 
    
    Since $\Conj(G)$ factors as $\Alg\to\Grp\to\Rack\to\Set$, it follows that $G$ and $\Conj(G)$ share the same underlying scheme; the distinction lies in the morphisms $(m,i,e)$ and $(\tr,\tr\inv,e)$ attached to this scheme. As before, $\Conj(G)$ is involutory if and only if $g^2\in Z(G)$ for all $g\in G$.
\end{ex}

\begin{ex}\label{ex:OLn}
Let $G$ be an algebraic group acting on a scheme $V$. The \emph{linear rack scheme} is the functor $\OL(G,V)\colon \Alg\to\Set$ that sends a $k$-algebra $R$ to the linear rack $\OL(G(R),V(R))$.

    In particular, we call $\OL_n\coloneq \OL(\GL_n,\mathbb{A}^n)$ the \emph{omni-linear rack variety}. (See Proposition \ref{prop:leib-of-ol} for the motivation behind this name and notation.) As a functor, $\OL_n$ sends every $k$-algebra $R$ to the omni-linear rack $\OL(\GL_n(R),R^n)$. Note that $\OL_n$ is an affine algebraic rack variety because $\GL_n$ and $\mathbb{A}^n$ are affine varieties. 
    Indeed, $\OL_n$ shares the same underlying scheme but not the same morphisms as $\mathrm{Aff}_n$, the linear algebraic group of affine transformations of $\mathbb{A}^n$.
\end{ex}

\subsection{}
Although the next example is uninteresting on its own, it is an important obstruction to adapting certain results about algebraic groups for algebraic racks.

\begin{ex}\label{ex:triv-1}
    Let $Q$ be a nonempty scheme. For every $k$-algebra $R$, view the set of $R$-points $Q(R)$ as a trivial quandle. Then $Q$ is an involutory algebraic quandle called a \emph{trivial quandle scheme}.
\end{ex}

\begin{rmk}\label{rmk:triv}
    A well-known theorem of Cartier (see, for example, \cite{milne}*{Thm.\ 3.23}) states that if $\operatorname{char}(k)=0$, then all affine algebraic groups over $k$ are reduced. Another classical result (see, for example, \cite{milne}*{Cor.\ 4.10}) states that every affine algebraic group is isomorphic to an algebraic subgroup of $\GL_n$ for some $n$. 
    
    The analogous claims for algebraic racks are far from true; in fact, Example \ref{ex:triv-1} shows that \emph{every} nonempty scheme can be made into an algebraic rack. Roughly speaking, this suggests that the condition of being an algebraic rack is much more lenient than that of being an algebraic group.
\end{rmk}

If we turn our attention to rack schemes $\Racks(\Sch)$ rather than algebraic racks $\Rack(\Sch)$, then Example \ref{ex:triv-1} generalizes to an even broader class of rack schemes.

\begin{ex}\label{ex:perm}
    Let $Q$ be any scheme, and let $\phi$ be an automorphism of $Q$. Define morphisms $\tr,\tr\inv\colon Q^2\to Q$ by
    \[
    \tr \coloneq \phi\circ\pi_2,\qquad \tr\inv \coloneq \phi\inv\circ\pi_1,
    \]
    where $\pi_1,\pi_2\colon Q^2\to Q$ are the projection morphisms. Then $Q$ is a rack scheme called a \emph{permutation rack scheme}. If $Q$ is nonempty and $\phi=\id_Q$, then we recover Example \ref{ex:triv-1}.
\end{ex}

\section{Corack algebras}\label{sec:corack}

Examples \ref{ex:conj} and \ref{ex:OLn} suggest that in order to understand the coordinate algebra $A$ of an affine algebraic rack $Q=\Spec(A)$, we have to study the pullbacks of the algebraic rack morphisms $\tr$, $\tr\inv$, and $e$. 
To that end, this section introduces \emph{(commutative) corack algebras}, which are to affine algebraic racks as commutative Hopf algebras are to affine algebraic groups (see Proposition \ref{prop:aff-racks}). Accordingly, corack algebras allow us to both state and prove Theorem \ref{thm:main}.

\subsection{}
Just as cogroup objects in cocartesian monoidal categories are formally dual to group objects in cartesian monoidal categories, we can also define a formal dual of pointed rack objects.

\begin{definition}
    Let $\C$ be a cocartesian monoidal category. A \emph{copointed corack object} or \emph{internal copointed corack} in $\C$ is a pointed rack object in the opposite category $\C\op$.
\end{definition}

Morphisms of copointed corack objects in $\C$ are also defined dually to those of pointed rack objects. Let $\Corack(\C)$ denote the category of copointed corack objects in $\C$.

\begin{definition}\label{def:corack}
    The category of \emph{(commutative) corack algebras} is $\Corack(\Alg)$. Its objects and morphisms are called \emph{corack algebras} and \emph{corack algebra homomorphisms}, respectively.
\end{definition}

For examples of corack algebras, see Examples \ref{ex:triv-2}--\ref{ex:ol-corack}.

\subsection{}
To make Definition \ref{def:corack} more concrete, let $A$ be a $k$-algebra with (commutative and associative) multiplication $\mu\colon A\otimes A\to A$, and recall that $\tau$ denotes the transposition $a\otimes b\mapsto b\otimes a$. 

A \emph{corack algebra structure} on $A$ consists of two $k$-algebra homomorphisms $\nabla,\nabla\inv\colon A\to A\otimes A$ called the \emph{corack operations} and a $k$-algebra homomorphism $\eps\colon A\to k$ called the \emph{counit}, \emph{coidentity}, or \emph{augmentation} that satisfy the following five identities induced by the diagrams in \eqref{ax:inv}--\eqref{ax:unit}:
\begin{enumerate}[(C1)]
    \item (Co-invertability) $(\mathord{\id}\otimes \mu)\circ(\tau\otimes\mu)\circ(\mathrm{\id}\otimes\nabla\inv\otimes\id)\circ(\nabla\circ\id)=(\mathord{\id}\otimes \id)=(\mu\otimes\id)\circ(\mu\otimes\tau)\circ(\mathord{\id}\otimes\nabla\otimes\id)\circ(\mathord{\id}\otimes\nabla\inv)$.
    \item\label{ax:cold} (Co-left distributivity) $(\mathord{\id}\otimes\nabla)\circ\nabla=(\mu\otimes\mathord{\id}\otimes\id)\circ(\mathord{\id}\otimes\tau\otimes\id)\circ(\nabla\otimes\nabla)\circ\nabla$.
    \item\label{ax:cord} (Co-right distributivity) $(\nabla\inv\otimes\id)\circ\nabla\inv=(\mathord{\id}\otimes\mathord{\id}\otimes\mu)\circ(\mathord{\id}\otimes\tau\otimes\id)\circ(\nabla\inv\otimes\nabla\inv)\circ\nabla\inv$.
    \item\label{ax:co-fixing} (Co-fixing) $(\eps\otimes \id)\circ\nabla=\id$.
    \item\label{ax:co-fixed} (Co-fixedness) $(\mathord{\id}\otimes \eps)\circ\nabla=\eps$.
\end{enumerate}
We also obtain the following two identities:
\begin{enumerate}[(C1')]
\setcounter{enumi}{3}
    \item\label{ax:co-inv-fixing} (Co-inverse fixing) $(\id\otimes\eps)\circ\nabla\inv=\id$.
    \item\label{ax:co-inv-fixed} (Co-inverse fixedness) $(\eps\otimes\id)\circ\nabla\inv=\eps$.
\end{enumerate}
We identify $k\otimes A\cong A\cong A\otimes k$ in identities \ref{ax:co-fixing}, \ref{ax:co-fixed}, \ref{ax:co-inv-fixing}, and \ref{ax:co-inv-fixed}. We identify $k$ as a $k$-subalgebra of $A$ in identities \ref{ax:co-fixed} and \ref{ax:co-inv-fixed}. 

\subsection{}
Similarly, corack algebra homomorphisms are $k$-algebra homomorphisms that are compatible with the structure morphisms $\nabla,\nabla\inv$, and $\eps$ in the obvious ways. For example (cf.\ Section \ref{sec:nabla}), if $\phi\colon Q\to R$ is a morphism of affine algebraic racks, then the pullback $\phi^\sharp\colon\O(R)\to\O(Q)$ satisfies
\[
    \nabla_Q^{\pm 1}\circ\phi^\sharp=(\phi^\sharp\otimes\phi^\sharp)\circ\nabla_R^{\pm 1},\qquad \eps_Q\circ \phi^\sharp=\eps_R.
    \]

\begin{rmk}
    Since the rack operations $\tr,\tr\inv$ are not assumed to be associative in the rack axioms, the corack operations $\nabla$ and $\nabla\inv$ are not assumed to be coassociative. (In fact, coassociativity only holds in trivial cases; see Proposition \ref{prop:coass}.) In particular, corack algebras are usually not coalgebras; this is why we chose not to denote the corack operations by $\Delta,\Delta\inv$.
\end{rmk}

\begin{notation}
    We use multi-index Sweedler notation for the corack operations $\nabla$ and $\nabla\inv$. For example, if $A$ is a corack algebra and $f\in A$, then we write $\nabla f=\sum\fo\otimes \ft$ and $\nabla \fo=\sum f_{(1)(1)}f_{(1)(2)}$.
    
    Note that non-coassociativity means that the expression $(\nabla\otimes\id)(\nabla f)=\sum f_{(1)(1)}\otimes f_{(1)(2)}\otimes \ft$ generally does not equal $(\mathord{\id}\otimes \nabla)(\nabla f)=\sum\fo\otimes f_{(2)(1)}\otimes f_{(2)(2)}$. This necessitates multi-indices, differing from the way one might use Sweedler notation for comultiplication in coalgebras.
\end{notation}

\begin{com}
    Definition \ref{def:corack} appears to be new. That said, similar algebraic structures called \emph{rack (bi)algebras} have been widely studied; see \cite{bardakov} for a list of references and \cite{alexandre} for applications of rack bialgebras to the theory of Lie racks and Leibniz algebras. 
    
    As the name suggests, rack bialgebras are typically assumed to be coassociative but not associative, while corack algebras are associative but only coassociative in trivial cases (see Proposition \ref{prop:coass}). Also, rack bialgebras are not assumed to carry a second operation $\nabla\inv$. 
    
    To give an analogy from group theory, rack bialgebras $k[Q]$ are to racks $Q$ as group algebras $k[G]$ are to groups $G$, while corack algebras $\O(Q)$ are to affine algebraic racks $Q=\Spec(\O(Q))$ as commutative Hopf algebras $\O(G)$ are to affine algebraic groups $G=\Spec(\O(G))$. We elaborate on the latter analogy in the following section.
\end{com}

\section{Affine algebraic racks}\label{sec:affine}
In the next four sections, we specialize our attention from algebraic racks $\Rack(\Sch)$ to affine algebraic racks $\Rack(\Aff)$. 

\subsection{Relationship with corack algebras}\label{sec:nabla}
The relationship between affine algebraic racks and corack algebras mirrors the relationship between commutative Hopf algebras and affine algebraic groups (see, for example, \cite{waterhouse}*{Sec.\ 1.4}). 
Namely, the morphisms $(\tr,\tr\inv,e)$ of an affine algebraic rack $Q=\Spec(\O(Q))$ induce a canonical corack algebra structure $(\nabla,\nabla\inv,\eps)\coloneq (\tr^\sharp,(\tr\inv)^\sharp,e^\sharp)$ on the coordinate algebra $\O(Q)$. 

In fact, due to the usual equivalence of categories $\Aff\cong(\Alg)\op$, we have the following.
\begin{prop}\label{prop:aff-racks}
    The assignments $Q\mapsto\O(Q)$ and $\phi\mapsto \phi^\sharp$ define an equivalence of categories $\Rack(\Aff)\bij(\Corack(\Alg))\op$.
\end{prop}

\begin{cor}[Cf.\ \cite{waterhouse}*{Sec.\ 2.1}]\label{cor:closed}
    Closed embeddings of affine algebraic racks $\phi\colon R\hookrightarrow Q$ correspond one-to-one to surjective corack algebra homomorphisms $\phi^\sharp\colon \O(Q)\twoheadrightarrow \O(R)$.
\end{cor}

\begin{com}
    Since commutative Hopf algebras are the same as cogroup objects in $\Alg$, Proposition \ref{prop:aff-racks} analogizes the fact that the category $\Grp(\Aff)$ of affine group schemes is equivalent to the opposite of the category of commutative Hopf algebras (see, for example, \cite{milne}*{Cor.\ 3.7}). By a general fact from universal algebra, this analogy extends to models of any Lawvere theory $\mathbb{T}$ in $\Aff$ and comodels of $\mathbb{T}$ in $\Alg$; see \cite{johnstone}*{Cor.\ 1.2.14} for details.

    Actually, this universal-algebraic perspective shows that $\Rack(\Sch)$ and $\Rack(\Aff)$ are both closed under fiber products and, in particular, base changes. These are constructed in exactly the same way as for, say, algebraic groups. Indeed, if $\mathbb{T}$ is an Lawvere theory and $\C$ is a concrete category with limits and sifted colimits of a certain shape, then the category of $\mathbb{T}$-models in $\C$ also has limits and sifted colimits of that shape; see \cite{mse}.
\end{com}

\subsection{Worked examples}
We compute the corack algebras of the affine algebraic racks discussed in Examples \ref{ex:conj}--\ref{ex:triv-1}. We will only compute the corack operation $\nabla$ in each example; computing the other corack operation $\nabla\inv$ is similar and left to the reader. As a warm-up, we begin with affine trivial quandle schemes.

\begin{ex}\label{ex:triv-2}
    Let $Q=\Spec(A)$ be any nonempty affine scheme, viewed as a trivial quandle scheme as in Example \ref{ex:triv-1}. Then the corack operation $\nabla$ of $A$ is given by
    \[
    \nabla f(x, y)=f(x\tr y)=f(y)
    \]
    for all functions $f\in A$ and points $x,y\in Q$. Hence,
    \[
    \nabla=1\otimes\id.
    \]
\end{ex}

\begin{rmk}\label{rmk:perm-nabla}
    If $(Q,\tr)$ is an affine rack scheme (as opposed to an affine algebraic rack, which is equipped with a unit $e$), then we can still define $(\nabla,\nabla\inv)\coloneq (\tr^\sharp,(\tr\inv)^\sharp)$. 
    
    In particular, if $(Q,\tr)$ is a permutation rack scheme with respect to $\phi\in\operatorname{Aut}(Q)$ as in Example \ref{ex:perm}, then the same calculation as in Example \ref{ex:triv-2} yields
    \[
    \nabla =1\otimes \phi^\sharp.
    \]
\end{rmk}

\begin{example}\label{ex:corack-conj}
    Let $G=\Spec(A)$ be an affine algebraic group. We compute the corack operation $\nabla$ of $A$ induced by the conjugation quandle scheme $\Conj(G)$.
    
    Recall that the group operations of $G$ induce a Hopf algebra structure $(\Delta,\eps,S)=(m^\sharp,e^\sharp,i^\sharp)$ on $A$; see, for example, \cite{waterhouse}*{Sec.\ 1.4}. (The notation is justified because the counit $\eps$ of $A$ as the Hopf algebra of $G$ coincides with the counit of $A$ as the corack algebra of $\Conj(G)$; both equal $e^\sharp$.) Given a function $\phi\in A$, we use Sweedler notation for the Hopf algebra comultiplication $\Delta$ by writing $\Delta \phi=\sum\phi_{(1)}\otimes \phi_{(2)}$. 
    
    For all functions $f\in A$ and elements $g,h\in G$, we have
    \begin{align*}
        \nabla f(g, h)&= f(g\tr h)\\
        &=f(ghg\inv)\\
        &= \Delta f(gh, g\inv) \\
        &=\sum \fo(gh)\ft(g\inv)\\
        &=\sum \Delta \fo(g, h)(S(\ft))(g)\\
        &=\sum f_{(1)(1)}(g)f_{(1)(2)}(h)(S(\ft))(g)\\
        &=\sum (f_{(1)(1)}S(\ft))(g)f_{(1)(2)}(h).
    \end{align*}
    Since $\Delta$ is coassociative, it follows that
    \[
    \nabla f=\sum \fo S(f_{(3)})\otimes \ft.
    \]
    
    Equivalently, let $\mu\colon A\otimes A\to A$ denote multiplication in $A$, and let $\tau$ denote the transposition $a\otimes b\mapsto b\otimes a$. Then
    \[
    \nabla=(\mu\otimes\id)\circ(\mathord{\id}\otimes S\otimes \id)\circ(\mathord{\id}\otimes\tau)\circ(\Delta\otimes\id)\circ\Delta.
    \]
\end{example}

\begin{example}\label{ex:ol-corack}
    Let $n\geq 0$, and consider the omni-linear rack variety $\OL_n$. As an affine variety, we have $\OL_n=\GL_n\times \mathbb{A}^n$, so the underlying $k$-algebra of the corack algebra $\O(\OL_n)$ is
    \begin{equation}\label{eq:ol-corack}
        \O(\OL_n)=\O(\GL_n)\otimes_k \O(\mathbb{A}_n)= k[s_{11},s_{12},\dots,s_{nn},{\det}\inv]\otimes_k k[t_1,\dots,t_n].
    \end{equation}
    The corack operation $\nabla$ acts on the first tensor factor $\O(\GL_n)$ as in Example \ref{ex:corack-conj}. On the second tensor factor $\O(\mathbb{A}^n)$, we have
    \[
    \nabla t_k((A,v),(B,w))=t_k((A,v)\tr(B,w))=t_k(ABA\inv,Aw)=(Aw)_k=\sum^n_{i=1}a_{ki}w_i
    \]
    for all elements $(A,v),(B,w)\in \OL_n$ and integers $1\leq k\leq n$. That is,
    \begin{equation}\label{eq:ol-nabla}
        \nabla t_k=\sum^n_{i=1}s_{ki}\otimes t_i.
    \end{equation}
\end{example}

\subsection{Properties of corack algebras}

By Proposition \ref{prop:aff-racks}, dualizing any theorem about affine algebraic racks $Q$ yields a theorem about corack algebras $\O(Q)$, and vice versa. 

For example, let $Q=\Spec(A)$ be an affine algebraic rack. Then for all $k$-algebras $R$ and $R$-points $x\in Q(R)$ (thought of as $k$-algebra homomorphisms $x^\sharp\colon A\to R$), the automorphisms $L_x^{\pm 1}\colon Q(R)\bij Q(R)$ from \eqref{eq:lx} induce automorphisms
\begin{align}
L_x^\sharp\colon R\otimes_k A &\bij R\otimes_k A\nonumber\\
    r\otimes f&\mapsto(r\otimes 1)\cdot ((x^\sharp\otimes\id_{A})\circ \nabla)(f)\label{eq:lx-sharp-full}
\end{align}
and
\begin{align}
(L_x\inv)^\sharp\colon A\otimes_k R &\bij A\otimes_k R\nonumber\\
    f\otimes r&\mapsto(1\otimes r)\cdot ((\id_{A}\otimes x^\sharp)\circ \nabla\inv)(f).\label{eq:lx-sharp-full-2}
\end{align}

\subsubsection{}
As another example, we show that corack algebras are not coassociative or cocommutative except in trivial cases. This stands in stark contrast to Hopf algebras of affine algebraic groups. Mirroring that setting, let us call a corack algebra $A$ \emph{coassociative} if $(\mathord{\id}\otimes\nabla)\circ\nabla=(\nabla\otimes\id)\circ\nabla$ and \emph{cocommutative} if $\tau\circ \nabla =\nabla$, where $\tau$ denotes the transposition $a\otimes b\mapsto b\otimes a$.

\begin{prop}\label{prop:coass}
    Let $A$ be a corack algebra. Then $A$ is coassociative if and only if $A$ is isomorphic to a corack algebra of the form computed in Example \ref{ex:triv-2}.
\end{prop}

\begin{proof}
    Let $Q\coloneq \Spec(A)$. By Proposition \ref{prop:aff-racks}, $Q$ is an affine algebraic rack, say with left-distributive rack operation $\tr$.  
    Proposition \ref{prop:aff-racks} also makes it sufficient to show that $\tr$ is associative if and only if $Q$ is a trivial quandle scheme, that is, $x\tr z=z$ for all $x,z\in Q$. Certainly, $\tr$ is associative if the latter condition holds. Conversely, if $\tr$ is associative, then left-distributivity \eqref{ax:ld} implies that
    \[
    L_{x\tr y}(z)=(x\tr y)\tr z=x\tr (y\tr z)=(x\tr y)\tr (x\tr z)=L_{x\tr y}(x\tr z).
    \]
    for all $x,y,z\in Q$. Since $L_{x\tr y}$ is invertible (see the discussion around \eqref{eq:lx}), the claim follows.
\end{proof}

\begin{prop}\label{prop:cocomm}
    Let $A$ be a corack algebra. The following are equivalent:
    \begin{enumerate}[(1)]
        \item\label{item:cocomm} $A$ is cocommutative. 
        \item\label{item:isotropic} $\nabla=\nabla\inv$.
        \item\label{item:singleton} $A$ has exactly one prime ideal.
    \end{enumerate}
    Under these conditions, $\nabla=1\otimes\id$ as in Example \ref{ex:triv-2}.
\end{prop}

\begin{proof}
    Once again, $Q\coloneq \Spec(A)$ is an affine algebraic rack, say with rack operations $\tr,\tr\inv$ and identity $e$. By Proposition \ref{prop:aff-racks}, condition \ref{item:cocomm} is equivalent to the condition that $\tr$ is commutative, and condition \ref{item:isotropic} is equivalent to the condition that $\tr=\tr\inv$. (In \cite{rack-roll}, rack objects for which these conditions hold are called \emph{symmetric} and \emph{isotropic}, respectively.) 
    
    Certainly, if $Q$ is a singleton (that is, if condition \ref{item:singleton} holds), then $\tr$ is commutative and $\tr=\tr\inv$. The converses follow from the fixing and fixedness axioms \ref{ax:fixing} and \ref{ax:fixed}. Namely, if $\tr$ is commutative, then
    \[
    x=e\tr x=x\tr e=e
    \]
    for all $x\in Q$. On the other hand, if $\tr=\tr\inv$, then the inverse fixedness identity \ref{ax:inv-fixed} implies that
    \[
    x=e\tr x= e\tr\inv x=e
    \]
    for all $x\in Q$. This proves the equivalence of all three conditions. The final claim follows from the fact that a singleton rack is necessarily a trivial quandle.
\end{proof}

\subsubsection{}
Apart from using affine algebraic racks to study corack algebras, we can also do the opposite. Given a pointed rack $Q$ (in $\Set$), recall (from, say, \cite{larosa2}, though it appears that the notion was originally introduced in \cite{elhamdadi} as ``the set of all stabilizers'') that the \emph{center} of $Q$ is the subset
\[
Z(Q)\coloneq \{x\in Q\mid L_x=\id\}.
\]
Note that $Z(Q)$ is actually a left ideal.
For example, if $G$ is a group, then $Z(\Conj(G))=Z(G)$. 

Inspired by the definition of the center of an algebraic group (see, for example, \cite{milne}*{Sec.\ 1.k}), we adapt this notion to algebraic racks. Recall from \eqref{eq:lx} that for all algebraic racks $Q$ and $k$-algebras $R$, each $R$-point $x\in Q(R)$ induces an isomorphism $L_x\colon Q(R)\bij Q(R)$.

\begin{definition}
    Let $Q\colon \Alg\to\Set$ be an algebraic rack. The \emph{center} of $Q$ is the subfunctor $Z(Q)\colon \Alg\to\Set$ defined by
    \[
    Z(Q)(R)\coloneq \{x\in Q(R)\mid L_x=\id_{Q(R)}\}
    \]
    for all $k$-algebras $R$.
\end{definition}

Note that $Z(Q)$ is a left ideal and trivial subquandle of $Q$. If $Q$ is affine, then we can generalize a classical result for centers of affine algebraic groups (see, for example, \cite{waterhouse}*{Ex.\ 3-14}).

\begin{prop}\label{prop:center}
    If $Q=\Spec(A)$ is an affine algebraic rack, then $Z(Q)$ is a closed affine algebraic left ideal of $Q$.
\end{prop}

\begin{proof}
    This is shown similarly to the analogous result for affine algebraic groups; see, for example, \cite{waterhouse}*{Ex.\ 3-14}). Namely, we have to show that $Z(Q)$ is representable as a functor. To that end, let $I$ be the intersection of all ideals $J$ of $A$ such that
    \[
    \nabla f\equiv 1\otimes f\pmod{J\otimes A}
    \]
    for all $f\in A$. 
    
    We claim that $A/I$ is a representing object for $Z(Q)$. By Proposition \ref{prop:aff-racks} and \eqref{eq:lx-sharp-full}, $Z(Q)$ is the largest subrack of $Q$ such that
    \[
    1\otimes f=L_x^\sharp(1\otimes f)=(x^\sharp\otimes\id_A)(\nabla f)
    \]
    for all $k$-algebras $R$, $R$-points $x\in Z(Q)(R)$, and functions $f\in A$. The claim follows.
\end{proof}

\begin{example}[Cf.\ \cite{waterhouse}*{Ex.\ 3-14}]\label{ex:center}
    Let $G$ be an affine algebraic group. Then $Z(\Conj(G))=Z(G)$ is a closed algebraic normal subgroup of $G$ and, accordingly, a closed algebraic left ideal of $\Conj(G)$.
\end{example}

\section{Preparation for Theorem \ref{thm:main}}\label{sec:pre-pf}
In this section, we give a preliminary discussion of the proof of Theorem \ref{thm:main} in Section \ref{sec:pf}.

\subsection{Strategy}
Our approach to the proof is inspired by Kinyon's proof of the analogous result for Lie racks in \cite{kinyon}*{Thm.\ 3.4}. 

\subsubsection{Lie racks}
Before outlining our approach, we briefly summarize Kinyon's proof. 
Given a real or complex Lie rack $Q$, let $\q\coloneq T_eQ$. Define a smooth map $\Ad\colon Q\to \GL(\q)$ by $x\mapsto \Ad_x\coloneq (dL_x)_e$, and consider its derivative $\ad\colon \q\to \mathfrak{gl}(\q)=\operatorname{End}(\q)$ defined by $X\mapsto \ad_X\coloneq (d\Ad)_e(X)$. We call $\Ad$ and $\ad$ the \emph{adjoint representations} of $Q$ and $\q$ on $\q$, respectively.

The left-distributivity axiom of racks states that
\[
(L_x\circ L_y)(z)=(L_{L_x(y)}\circ L_x)(z)
\]
for all $x,y,z\in Q$. 
Differentiating at $e\in Q$ twice, first with respect to $z$ and then with respect to $y$, yields
\[
(\Ad_x \circ\ad_Y)(Z)=(\ad_{\Ad_x(Y)}\circ\Ad_x(Z))
\]
for all $x\in Q$ and $Y,Z\in \q$. Differentiating at $e$ wih respect to $x$ yields the Jacobi identity in terms of $\ad$, showing that $\q$ has a Leibniz algebra structure given by $[X,Y]\coloneq \ad_X(Y)$.

\subsubsection{Outline}\label{sec:broad-outline}
In the algebraic setting, we also work with adjoint representations. Later in this section (see Definition \ref{def:ad2}), we define adjoint representations in the same way that Kinyon does, replacing analytic derivatives with tangent maps of affine schemes. 

Given an affine algebraic rack $Q=\Spec(A)$, let $\q\coloneq \Der_k(A,k)$. In Section \ref{sec:pf}, we begin by showing that the adjoint representation of $\q$ agrees with convolution $[\cdot,\cdot]$; see Proposition \ref{prop:adj} and Remark \ref{rmk:recover}. 
In lieu of partial differentation, we verify the Jacobi identity using the convolution formula and the corack identities of $A$; see Proposition \ref{prop:jacobi}. Compatibility with morphisms is straightforward from there; see Proposition \ref{prop:functor}.

\subsection{Tangent spaces}
We discuss tangent spaces and tangent maps of affine schemes; see \cite{milne}*{App.\ A.h} and \cite{borel}*{AG.16} for details. 
Henceforth, let $Q=\Spec(A)$ be an affine algebraic rack, and recall from Section \ref{sec:nabla} that $Q$ induces a corack algebra structure $(\nabla,\nabla\inv,\eps)$ on $A$.

\subsubsection{}We construct the Leibniz algebra of $Q$. Our approach is modeled after the construction of the Lie algebra of an algebraic group in \cite{milne}*{Sec.\ 10b}; see also \cite{waterhouse}*{Sec.\ 12.2} and \cite{borel}*{Sec.\ 3.5}.

Since the counit $\eps\colon A\to k$ induces an $A$-module structure on $k$, recall (from, say, \cite{milne}*{Sec.\ 10e}) that a \emph{$k$-derivation} $D\colon A\to k$ is a $k$-linear map satisfying the \emph{Leibniz rule} for all $f,g\in A$:
\[
D(fg)=D(f)\eps(g)+\eps(f)D(g).
\]
The vector space of all such derivations is denoted by $\Der_k(A,k)$.

\begin{definition}
    The \emph{left Leibniz algebra of $Q$} is the vector space
    \[
    \q\coloneq \Der_k(A,k)
    \]
    equipped with the Leibniz bracket $[\cdot,\cdot]\colon \q^2\to \q$ defined by 
    \[
    [D,E]\coloneq (D\otimes E)\circ\nabla,
    \]
    the \emph{convolution} operation with respect to $\nabla$.

    Similarly, the \emph{right Leibniz algebra of $Q$} is the pair $(\q,\{\cdot,\cdot\})$, where
    \[
    \{D,E\}\coloneq (D\otimes E)\circ\nabla\inv.
    \]
\end{definition}

\begin{rmk}
    Given a function $f\in A$, write $\nabla f=\sum \fo\otimes\ft$. Then
\[
[D,E](f)=\sum D(\fo)E(\ft),
\]
and similarly for $\{D,E\}(f)$. This makes it clear that $[\cdot,\cdot]$ and $\{\cdot,\cdot\}$ are $k$-bilinear.
\end{rmk}

\subsubsection{}
We give an alternative interpretation of $\q$. 
Let $T_e Q$ denote the tangent space of $Q$ at the identity $e$, and recall that $k[\delta]\coloneq k[\delta]/(\delta^2)$ denotes the $k$-algebra of dual numbers.
Using one of the equivalent characterizations of tangent spaces of schemes (see, for example, \cite{milne}*{A.51}), we can identify $T_eQ$ with the set of $k$-algebra homomorphisms $ A\to k[\delta]$ whose composite with the canonical map $k[\delta]\to k$ equals the counit $\eps\colon A\to k$. This provides the following isomorphism.

\begin{prop}\label{prop:tangent}
    Define a map $\partial\colon \q\to T_eQ$ by assigning to each $k$-derivation $D\colon A\to k$ the $k$-algebra homomorphism     \[
    \partial D\coloneq \partial(D)\coloneq \eps+\delta D\colon A\to k[\delta].
    \]
    Then $\partial$ is a bijection that induces a $k$-vector space structure on $T_e Q$.
\end{prop}

\begin{proof}
    The claim is proven identically to the analogous result for algebraic groups; see, for example, \cite{milne}*{Prop.\ 10.28}.
\end{proof}

\subsubsection{}
Similarly, recall (from, say, \cite{milne}*{Prop.\ 10.28}) that for every vector space $V$ over $k$, we have isomorphisms
\begin{align}
    \Der_k(\O(\GL(V)),k)&\bij T_{\id}\mathord{\GL}(V)& \mathfrak{gl}(V)=\operatorname{End}(V) &\bij T_{\id}\mathord{\GL}(V)\nonumber\\
    F& \mapsto \varepsilon_{\O(\GL(V))}+ \delta F, &   f&\mapsto \id_{k[\delta]\otimes_k V }+\delta f. \label{eq:glq}
\end{align}
In the isomorphism on the left, we identify $T_{\id}\mathord{\GL}(V)$ as the vector space of $k$-algebra homomorphisms $\O(\GL(V))\to k[\delta]$ whose composite with $k[\delta]\to k$ is the counit $\varepsilon_{\O(\GL(V))}$ of the Hopf algebra $\O(\GL(V))$. On the right, we dually identify $T_{\id}\mathord{\GL}(V)$ as the subgroup 
\[
T_{\id}\mathord{\GL}(V)=\ker(\GL(V)(k[\delta])\to \GL(V)(k))\leq \GL(V)(k[\delta]).
\]
(See, for example, \cite{milne}*{10.7}.) 

In particular, to pass from $\Der_k(\O(\GL(\q)),k)$ to $\mathfrak{gl}(\q)$ during the proof of Theorem \ref{thm:main} using \eqref{eq:glq}, we will need to pass between these two equivalent interpretations of $T_{\id}\mathord{\GL}(\q)$. See Section \ref{sec:sketch} for an overview of this argument.

\subsubsection{Tangent maps}
We recall the definition of the derivative of a morphism of schemes; see \cite{borel}*{AG.16} for details. 

Given a morphism of schemes $\phi\colon X\to Y$ and a point $x\in X$, let $\phi^\sharp\colon \O(Y)\to\O(X)$ be the pullback $k$-algebra homomorphism. Identify $T_xX$ and $T_{\phi(x)}Y$ with the respective spaces of $k$-derivations. Then the \emph{derivative}, \emph{tangent map}, or \emph{differential} of $\phi$ at $x$ is the $k$-linear map $(d\phi)_x\colon T_xX\to T_{\phi(x)}Y$  defined by 
\[
(d\phi)_x(D)\coloneq D\circ\phi^\sharp
\]
for all $k$-derivations $D\in \Der_k(\O_{X,x},k)$. Differentiation is functorial in the obvious sense; in analogy with differential calculus, this fact is called the \emph{chain rule}.

\subsection{Adjoint representations} In analogy with Kinyon's proof, we define the \emph{adjoint representations} of $Q$ and $\q$ on $\q$.

Recall from \eqref{eq:lx} that every point $x\in Q$ induces algebraic rack automorphisms $L_x,L_x\inv\colon Q\bij Q$. Due to the fixedness axiom (that is, the second diagram in \eqref{ax:unit}) and the fact that differentiation is functorial, the tangent maps $(dL_x)_e,(dL_x\inv)_e$ are $k$-linear automorphisms of $\q$. This allows us to define the following.
\begin{definition}\label{def:ad2}
    The \emph{adjoint representation} of $Q$ on $\q$ is the morphism of schemes 
    $\Ad\colon Q\to\GL(\q)$ defined by sending each element $x\in Q$ to the $k$-linear automorphism
    \[
    \Ad_x \coloneq  (dL_x)_e\in\GL(\q).
    \]
    The derivative $\ad\coloneq (d\Ad)_e\colon \q\to \Der_k(\O(\GL(\q)),k)$ is called the \emph{adjoint representation} of $\q$ on itself. Given a $k$-derivation $D$, write $\ad_D\coloneq\ad(D)$.

    Similarly, define the \emph{dual adjoint representations} of $Q$ and $\q$ via
    \[
    \Ad\inv_x\coloneq (dL_x\inv)_e,\qquad \ad\inv\coloneq (d\Ad\inv)_e.
    \]
\end{definition}

\subsubsection{}\label{sec:sketch}
For the purposes of proving Theorem \ref{thm:main}, Definition \ref{def:ad2} is problematic because it defines the codomain of $\ad^{\pm 1}$ to be $\Der_k(\O(\GL(\q)),k)$ rather than $\mathfrak{gl}(\q)$. To find the element of $\mathfrak{gl}(\q)$ corresponding to $\ad_D^{\pm 1}\in \Der_k(\O(\GL(\q)),k)$ for each $k$-derivation $D\in\q$, we employ the isomorphisms in Proposition \ref{prop:tangent} and \eqref{eq:glq}. Section \ref{sec:convo} is dedicated to this argument.

Here is an overview of the argument in Section \ref{sec:convo}. Given $D\in\q$, we lift $D$ to a $k$-algebra homomorphism $\partial D\colon A\to k[\delta]$ as in Proposition \ref{prop:tangent}. Viewing $\partial D$ as a $k[\delta]$-point of $Q$ gives us $k[\delta]$-points $\Ad_{\partial D}\in\GL(k[\delta]\otimes_k\q )$ and $\Ad_{\partial D}\inv\in\GL(\q\otimes_k k[\delta])$ of $\GL(\q)$. 

Identify $k[\delta]\otimes_k\q$ and $\q\otimes_k k[\delta]$ with $\q[\delta]/(\delta^2)$.
Then the second isomorphism in \eqref{eq:glq} states that for all $k$-derivations $E\in\q$, the element of $\q$ corresponding to $\ad_D(E)$ (resp.\ $\ad_D\inv(E)$) is precisely the coefficient of $\delta$ in $\Ad_{\partial D}(1\otimes E)$ (resp.\ $\Ad_{\partial D}\inv(E\otimes 1)$) in $\q[\delta]/(\delta^2)$. We show that this coefficient is the convolution $[D,E]$ (resp.\ $\{E,D\}$).

\section{Proof of Theorem \ref{thm:main}}\label{sec:pf}
This section is dedicated to proving Theorem \ref{thm:main}. See Section \ref{sec:broad-outline} for an outline of the proof.

\subsection{Convolution versus $\ad$}\label{sec:convo}
Our first goal is to show that $\ad$ agrees with the Leibniz bracket $D\mapsto [D,\cdot]$ when considered as a map $\q\to\mathfrak{gl}(\q)$. Along with providing an explicit formula for $\ad$, this will show \emph{a fortiori} that $\q$ is closed under convolution. It will also prove the statement in the second paragraph of Theorem \ref{thm:main}; see Remark \ref{rmk:recover}. Our method follows the outline in Section \ref{sec:sketch}.

\subsubsection{}
In the following, fix arbitrary $k$-derivations $D,E\in\q$. Our proofs of Propositions \ref{prop:adj} and \ref{prop:inv} will use the $k$-linear maps $\psi_D,\phi_D\colon A\to A$ defined by
\begin{equation}\label{eq:psi-d}
    \psi_D\coloneq ( D\otimes\id)\circ\nabla,\qquad \phi_D\coloneq (\id\otimes D)\circ\nabla\inv.
\end{equation}
Note in particular that
\begin{equation}\label{eq:comm}
    E\circ \psi_D = [D,E],\qquad E\circ\phi_D=\{E,D\}
\end{equation}
because $E$ is $k$-linear. 

\subsubsection{}
Define $\partial D$ as in Proposition \ref{prop:tangent}, and view $\partial D$
as a $k[\delta]$-point of $Q$. Also, consider the $k$-algebra
\[
    A[\delta]\coloneq A[\delta]/(\delta^2).
\]
Via the identifications $k[\delta]\otimes_kA \bij A[\delta]$ and $A\otimes_k k[\delta] \bij A[\delta]$,
we obtain from \eqref{eq:lx-sharp-full} and \eqref{eq:lx-sharp-full-2} two automorphisms $L_{\partial D}^\sharp,(L_{\partial D}\inv)^\sharp$ of $A[\delta]$ defined by
\[
L_{\partial D}^\sharp(\delta^n f)= \delta^n(\partial D\otimes \id_A)(\nabla f),\qquad (L_{\partial D}\inv)^\sharp(\delta^n f)=\delta^n(\id_A\otimes\partial D)(\nabla\inv f)
\]
for all $f\in A$. The following gives alternative formulas for these automorphisms.

\begin{lemma}\label{lem:ld-full}
    For all $k$-derivations $D\in\q$ and functions $f\in A$, we have \[L_{\partial D}^\sharp(\delta^n f)=\delta^n(f +\delta \psi_D(f)),\qquad (L_{\partial D}\inv)^\sharp(\delta^n f)=\delta^n(f +\delta \phi_D(f)).\]
\end{lemma}

\begin{proof}
    For all $f\in A$, write $\nabla f=\sum \fo\otimes\ft$. Then we compute
    \begin{align*}
    L_{\partial D}^\sharp(\delta^n f)&=\delta^n\sum \partial D(\fo)\ft\\
    &=\delta^n\sum(\eps(\fo)+\delta D(\fo))\ft \\
    &= \delta^n\left(\sum \eps(\fo)\ft+\delta\sum D(\fo)\ft\right)\\
    &=\delta^n(f+\delta \psi_D(f)), 
\end{align*}
where in the last equality we have used the co-fixing identity \ref{ax:co-fixing}. The calculation for $(L_{\partial D}\inv)^\sharp$ is similar and uses the co-inverse fixing identity \ref{ax:co-inv-fixing}.
\end{proof}

\subsubsection{}
Similarly, since $\partial D$ is a $k[\delta]$-point of $Q$ and $\Ad^{\pm 1}\colon Q\to\GL(\q)$ are morphisms of schemes, we obtain $k[\delta]$-points of $\GL(\q)$ \[\Ad_{\partial D}\in\GL(k[\delta]\otimes_k \q),\qquad \Ad_{\partial D}\inv\in\GL(\q\otimes_k k[\delta]).\] Identify $k[\delta]\otimes_kA \bij A[\delta]$ and $A\otimes_kk[\delta] \bij A[\delta]$. Then the definition of the derivative shows for all $k$-derivations $E\in\q$, we have 
\begin{equation}\label{eq:adj-partial}
    \Ad_{\partial D}( 1\otimes E)= E\circ L_{\partial D}^\sharp ,\qquad \Ad_{\partial D}\inv( E\otimes 1)= E\circ (L_{\partial D}\inv)^\sharp
\end{equation}
as maps $A[\delta]\to k[\delta]$. We use this to show the following. \begin{prop}\label{prop:adj}
    When considered as a map $\ad\colon\q\to\mathfrak{gl}(\q)$, the adjoint representation of $\q$ is precisely the map $D\mapsto [D,\cdot]$. That is, $\ad_D(E)=[D,E]$ for all $k$-derivations $D,E\in\q$. Similarly, $\ad\inv_D(E)=\{E,D\}$ if $\ad\inv$ is considered as a map $\q\to\mathfrak{gl}(\q)$. 
    
    In particular, $\q$ is closed under $[\cdot,\cdot]$ and $\{\cdot,\cdot\}$. 
\end{prop}

\begin{proof}
    We prove the claim for $\ad$ and $[\cdot,\cdot]$; the claim for $\ad\inv$ and $\{\cdot,\cdot\}$ is shown similarly.
    
    Recall that $\ad=(d\Ad)_e$. Identify $\ad_D\in\Der_k(\O(\GL(V)),k)$ as an element of $\mathfrak{gl}(\q)$ under the isomorphisms in \eqref{eq:glq} and the discussion following it. Then $\ad_D(E)\in\q$ equals the coefficient of $\delta$ in $\Ad_{\partial D}(1\otimes E)$, where $\Ad_{\partial D}$ is viewed as a $k[\delta]$-point of $\GL(\q)$. To compute this coefficient, combine Lemma \ref{lem:ld-full} with \eqref{eq:adj-partial} to obtain
    \[
    \Ad_{\partial D}( 1\otimes E)(f)=E(f)+\delta E(\psi_D(f))
    \]
    for all functions $f\in A$. Hence, \eqref{eq:comm} shows that
    \[
    \ad_D(E)=E\circ\psi_D=[D,E],
    \]
    as desired.
\end{proof}

\begin{rmk}\label{rmk:recover}
    If $G$ is an affine algebraic group and $Q=\Conj(G)$ is its conjugation quandle scheme, then $(\q,[\cdot,\cdot])$ is the Lie algebra of $G$. To see this, note that $Q$ and $G$ have the same underlying scheme and identity $e$, so they also have the same tangent space $\q$ at $e$. By the definition of $\Conj(G)$, the adjoint representations $\Ad$ of $Q$ and $\ad$ of $\q$ are therefore the same as those of $G$ and its Lie algebra (see, for example, \cite{milne}*{Rem.\ 10.24}). Hence, the claim follows from Proposition \ref{prop:adj}.\footnote{Alternatively, one can use Example \ref{ex:corack-conj} and the Hopf algebra axioms to directly (if tediously) verify that the Leibniz bracket $[D,E]=(D\otimes E)\circ\nabla$ agrees with the usual Lie bracket $(D\otimes E-E\otimes D)\circ\Delta$ of $\q=\operatorname{Lie}(G)$ (see \cite{waterhouse}*{Sec.\ 12.2}); we leave the details to the reader. As a hint, first deduce the identity $\eps\circ S=\eps$ from the axiom $\mu\circ (S\otimes\id)\circ\Delta=\eta\circ\eps$ by post-composing with $\eps$, and then deduce from these two equations that $D\circ S=-D$.}

    Similarly, if $k=\R$ or $\mathbb{C}$ and $Q(k)$ is an algebraic Lie rack, then our definitions of $\Ad$ and $\ad$ coincide with the ones for Lie racks in \cite{kinyon}, so $(\q,[\cdot,\cdot])$ is the Leibniz algebra of $Q(k)$ as a Lie rack.
\end{rmk}

\subsection{Jacobi identity}
Next, we verify that $[\cdot,\cdot]$ and $\{\cdot,\cdot\}$ satisfy the left and right Jacobi identities, respectively. 
Similarly to how the analogous result for Lie racks \cite{kinyon}*{Thm.\ 3.4} relies on the left-distributivity axiom \ref{ax:left} of racks, our verification of the left Jacobi identity for $[\cdot,\cdot]$ relies on the co-left distributivity identity \ref{ax:cold} and the co-fixing identity \ref{ax:co-fixing}. The verification of the right Jacobi identity for $\{\cdot,\cdot\}$ is similar; we leave it to the reader to adapt the argument using the co-right distributivity identity \ref{ax:cord} and the co-inverse fixing identity \ref{ax:co-inv-fixing}.

\subsubsection{}
In the following, let $X,Y,Z\in\q$ be $k$-derivations, and let $f\in A$. To avoid using double indices, write $\nabla f=\sum u\otimes v$, $\nabla u=\sum \uo\otimes\ut$, and $\nabla v= \sum\vo\otimes\vt$. We begin with two direct calculations involving the right-hand side of the co-left distributivity identity \ref{ax:cold}. 

\begin{lemma}\label{lem:sums}
    As an element of $k$, the expression \[S\coloneq ((X\otimes Y\otimes Z)\circ(\mu\otimes\mathord{\id}\otimes\id)\circ(\mathord{\id}\otimes\tau\otimes\id)\circ(\nabla\otimes\nabla)\circ\nabla)(f),\] which comes from \ref{ax:cold}, can be written as the sum of two expressions
    \[
    S_1\coloneq \sum X(\uo) \eps(\vo)Y(\ut) Z(\vt)
    \]
    and
    \[
    S_2\coloneq\sum\eps(\uo)X(\vo)Y(\ut)Z(\vt) .
    \]
\end{lemma}

\begin{proof}
    Since $X$ is a $k$-derivation, we have
    \begin{align*}
        S &= (X\otimes Y\otimes Z)\left(\sum\uo\vo\otimes \ut\otimes \vt\right )\\
        &=\sum X(\uo\vo) Y(\ut)Z(\ut)\\
        &= \sum(X(\uo)\eps(\vo)+\eps(\uo)X(\vo))Y(\ut)Z(\ut)\\
        &= S_1+S_2,
    \end{align*}
    as desired.
\end{proof}

\begin{lemma}\label{lem:s1}
    Let $S_1$ and $S_2$ be the sums defined in Lemma \ref{lem:sums}. Then
    \[
    S_1=[[X,Y],Z](f),\qquad S_2=[Y,[X,Z]](f) .
    \]
\end{lemma}

\begin{proof}
    We have
    \begin{align*}
        S_1 &= \sum X(\uo)Y(\ut)\eps(\vo)Z(\vt) \\
        &= \sum (X\otimes Y)(\nabla u) Z(\eps(\vo)\vt)&\text{because $Z$ is $k$-linear}\\
        &=\sum [X,Y](u) (Z\circ (\eps\otimes\id)\circ \nabla)(v)\\
        &= \sum [X,Y](u) Z(v)&\text{by the co-fixing identity \ref{ax:co-fixing}}\\
        &=([X,Y]\otimes Z)\left(\sum u\otimes v\right)\\
        &= ([X,Y]\otimes Z)(\nabla f)\\
        &= [[X,Y],Z](f),
    \end{align*}
    as desired. The calculation for $S_2$ is similar.
\end{proof}

\subsubsection{}
Now we can finally verify that $(\q,[\cdot,\cdot])$ is a left Leibniz algebra.

\begin{prop}\label{prop:jacobi}
    The convolution operation $[\cdot,\cdot]$ satisfies the left Jacobi identity.
\end{prop}

\begin{proof}
    Let $S$, $S_1$, and $S_2$ be the expressions defined in Lemma \ref{lem:sums}. Using the same notation as before, we compute
    \begin{align*}
        [X,[Y,Z]](f) &= (X\otimes [Y,Z])(\nabla f)\\
        &=\sum X(u)(Y\otimes Z)(\nabla v)\\
        &= \sum X(u)Y(\vo)Z(\vt)\\
        &= (X\otimes Y\otimes Z)\left(\sum u\otimes\vo\otimes\vt\right)\\
        &= (X\otimes Y\otimes Z)((\id\otimes\nabla)\circ \nabla )(f)\\
        &= S =S_1+S_2 &\text{by \ref{ax:cold} and  Lemma \ref{lem:sums}}\\
        &=([[X,Y],Z] + [Y,[X,Z]])(f), &\text{by Lemma \ref{lem:s1}}
    \end{align*}
    as desired. Since $f\in A$ was arbitrary, the proof is complete.
\end{proof}

\subsection{Action on morphisms}
To complete the proof of Theorem \ref{thm:main}, it remains to show that the assignment $Q\mapsto\q$ is compatible with morphisms. This is a straightforward verification, but we provide it below for the reader's convenience.

\begin{prop}\label{prop:functor}
    The assignments $Q\mapsto(\q,[\cdot,\cdot])$ (resp.\ $Q\mapsto(\q,\{\cdot,\cdot]\}$) and $\phi\mapsto \phi'\coloneq (d\phi)_e$ define a covariant functor $\Rack(\Aff)\to\leib$ (resp.\ $\Rack(\Aff)\to\Rleib$).
\end{prop}

\begin{proof}
    Once again, we prove the claim for $\Rack(\Aff)\to\leib$; the proof for $\Rack(\Aff)\to\Rleib$ is similar.
    
    Functoriality follows from the chain rule. Therefore, we only have to show that if $\phi\colon Q\to R$ is a morphism of affine algebraic racks, then $\phi'$ is a Leibniz algebra homomorphism $\q\to\mathfrak{r}$, where $(\q,[\cdot,\cdot]_\q)$ and $(\mathfrak{r},[\cdot,\cdot]_\mathfrak{r})$ are the left Leibniz algebras of $Q$ and $R$, respectively. 
    
    Recall that $\phi'$ is defined by
    \[
    \phi'(D)= D\circ\phi^\sharp
    \]
    for all $k$-derivations $D\in\q$.
    Certainly, $\phi'$ is $k$-linear. That $\phi'(\q)\subseteq\mathfrak{r}$ follows from the identification of $\q$ and $\mathfrak{r}$ as tangent spaces at the identity via Proposition \ref{prop:tangent} (though verifying this containment directly is straightforward). 
    
    Thus, it only remains to verify that $\phi'$ preserves the Leibniz bracket. 
    Recall that 
    \[
    \nabla_Q\circ\phi^\sharp=(\phi^\sharp\otimes\phi^\sharp)\circ\nabla_R
    \]
    because $\phi$ is a morphism of algebraic racks. 
    Hence, for all $k$-derivations $D,E\in\q$, we have
    \begin{align*}
        \phi'([D,E]_\q)&=[D,E]_\q\circ \phi^\sharp\\
        &=(D\otimes E)\circ\nabla_Q\circ\phi^\sharp\\
        &=(D\otimes E)\circ(\phi^\sharp\otimes\phi^\sharp)\circ \nabla_R\\
        &=((D\circ\phi^\sharp)\otimes (E\circ\phi^\sharp))\circ\nabla_R\\
        &= (\phi'(D)\otimes\phi'(E))\circ\nabla_R \\
        &=[\phi'(D),\phi'(E)]_\mathfrak{r},
    \end{align*}
    as desired.
\end{proof}

This completes the proof of Theorem \ref{thm:main}.

\begin{cor}\label{cor:invariant}
    The left and right Leibniz algebras $(\q,[\cdot,\cdot]),(\q,\{\cdot,\cdot\})$ are invariants of affine algebraic racks.
\end{cor}

\begin{rmk}
    Since racks are used to develop invariants of knots, 3-manifolds, and various algebro-geometric objects (see Section \ref{sec:open}), Corollary \ref{cor:invariant} may allow for Leibniz algebra-valued invariants to be developed for the same objects.
\end{rmk}

\section{From rack structure to Leibniz algebra structure}\label{sec:from}
As applications of Theorem \ref{thm:main}, we provide several results relating the structure an affine algebraic rack to the structure of its left Leibniz algebra. These results recover the analogous results for the Lie algebras of affine algebraic groups. We leave it to the reader to adapt these results to the right Leibniz algebras of affine algebaic racks. 

In the following, let $Q=\Spec(A)$ and $R=\Spec(B)$ be affine algebraic racks with left Leibniz algebras $\q$ and $\mathfrak{r}$, respectively.

\subsection{Subracks and ideals}
The following results relate closed subracks and left ideals of affine algebraic racks to Leibniz subalgebras and left ideals of their left Leibniz algebras, respectively. These results are shown more or less identically to the analogous results for Lie algebras of affine algebraic groups, but we provide the proofs for the reader's convenience.

\subsubsection{}
Recall that if $(\mathfrak{g},[\cdot,\cdot])$ is a left Leibniz algebra, then a subspace $\mathfrak{h}\leq \mathfrak{g}$ is called a \emph{Leibniz subalgebra} if $[\mathfrak{h},\mathfrak{h}]\subseteq\mathfrak{h}$ and a \emph{left ideal} if $[\mathfrak{g},\mathfrak{h}]\subseteq\mathfrak{h}$. On the other hand, we say that $\mathfrak{g}$ is \emph{abelian} if $[\mathfrak{g},\mathfrak{g}]=0$.

\begin{prop}\label{prop:subalg}
    Let $\phi\colon R\hookrightarrow Q$ be a closed embedding of affine algebraic racks. Then ${(d\phi)_e\colon\mathfrak{r}\to\q}$ embeds $\mathfrak{r}$ into $\mathfrak{q}$ as a Leibniz subalgebra.
\end{prop}

\begin{proof}
    If $D\in\ker((d\phi)_e)$, then
    \[
    0\equiv (d\phi)_e(D)=D\circ\phi^\sharp
    \]
    as a map $A\to k$.
    Recall from Corollary \ref{cor:closed} that $\phi^\sharp\colon A\twoheadrightarrow B$ is surjective. It follows that $D\equiv 0$, so by Proposition \ref{prop:functor}, $(d\phi)_e$ is an embedding of Leibniz algebras.
\end{proof}

\begin{prop}\label{prop:ideal}
    If $R$ is a left ideal of $Q$, then $\mathfrak{r}$ is a left ideal of $\mathfrak{q}$.
\end{prop}

\begin{proof}
    Consider the adjoint representations $\Ad$ and $\ad$ from Definition \ref{def:ad2}.
    By Proposition \ref{prop:adj}, it suffices to show that for all $k$-derivations $D\in\q$, the adjoint map $\ad_D\colon\q\to\q$ restricts to an endomorphism of $\mathfrak{r}$.
     
    By assumption, for all points $x\in Q$, the automorphism $L_x$ from \eqref{eq:lx} restricts to an automorphism of $R$, so its derivative $(dL_x)_e$ restricts to an automorphism of $\mathfrak{r}$. This allows us to view $\Ad$ as a map $Q\to \GL(\mathfrak{r})$. Hence, $\ad$ can be viewed as a map $\q\to\mathfrak{gl}(\mathfrak{r})$. Since $\mathfrak{gl}(\mathfrak{r})=\operatorname{End}(\mathfrak{r})$, this completes the proof.
\end{proof}

\begin{ex}
    Let $\phi\colon H\hookrightarrow G$ be a closed embedding of affine algebraic groups. If ${Q\coloneq \Conj(G)}$ and $R\coloneq \Conj(H)$ are the corresponding conjugation quandle schemes, then it is easy to see that $\phi\colon R\hookrightarrow Q$ is also a closed embedding of affine algebraic quandles.
    
 Thus, Proposition \ref{prop:subalg} recovers the well-known fact (see, for example, \cite{waterhouse}*{Sec.\ 12.2}) that $(d\phi)_e$ embeds $\mathfrak{r}=\operatorname{Lie}(H)$ into $\mathfrak{q}=\operatorname{Lie}(G)$. Likewise, Proposition \ref{prop:ideal} recovers the fact (see, for example, \cite{waterhouse}*{Ex.\ 12-8}) that if $\phi(H)$ is normal in $G$, then $(d\phi)_e(\mathfrak{r})$ is a two-sided ideal of $\mathfrak{q}$. This is because every left ideal of a Lie algebra is necessarily a two-sided ideal.
\end{ex}

\subsubsection{}
Recall that the \emph{left center} of a left Leibniz algebra $(\mathfrak{g},[\cdot,\cdot])$ is the subset
\[
Z_\ell(\mathfrak{g})\coloneq\{X\in \mathfrak{g}\mid [X,\mathfrak{g}]=0\},
\]
which is an abelian left ideal of $\mathfrak{g}$.\footnote{One may analogously define a \emph{right center} $Z_r(\mathfrak{g})$, but $Z_r(\mathfrak{g})$ is not a left Leibniz subalgebra of $\mathfrak{g}$ in general.} Note that if $\mathfrak{g}$ is a Lie algebra, then $Z_\ell(\mathfrak{g})$ is just the center of $\mathfrak{g}$ as a Lie algebra. 

The following is a consequence of Propositions \ref{prop:center}, \ref{prop:adj}, and \ref{prop:ideal}.
\begin{cor}\label{cor:center}
    Let $Q$ be an affine algebraic rack with Leibniz algebra $\q$. Then the functor in Theorem \ref{thm:main} sends the center $Z(Q)$ to a left ideal of $\q$ contained in the left center $Z_\ell(\q)$.
\end{cor}

\begin{rmk}
    Corollary \ref{cor:center} has also been shown for Lie racks; see \cite{larosa2}*{Sec.\ 4}.
\end{rmk}

In particular, Example \ref{ex:center}, Remark \ref{rmk:recover}, and Corollary \ref{cor:center} recover the fact that if $G$ is an affine algebraic group, then $\operatorname{Lie}(Z(G))$ is an ideal of $ \operatorname{Lie}(G)$ contained in the center of $\operatorname{Lie}(G)$. Note that $\operatorname{Lie}(Z(G))$ may be strictly smaller than the center; for example, see \cite{milne}*{10.35}.

\subsection{Abelian conditions}
Finally, we provide sufficient conditions under which $\q$ is abelian. The first condition generalizes the fact that Lie algebras of abelian affine algebraic groups are abelian (see, for example, \cite{waterhouse}*{Ex.\ 12-3}).

\begin{prop}\label{prop:inv}
    If $Q=\Spec(A)$ is involutory and $\operatorname{char}(k)\neq 2$, then $\q$ is abelian.
\end{prop}

\begin{proof}
    Fix a $k$-derivation $D\in\q$, define $\psi_D\colon A\to A$ as in \eqref{eq:psi-d}, and recall that $k[\delta]= k[\delta]/(\delta^2)$ denotes the dual numbers. By \eqref{eq:comm}, it suffices to show that $\psi_D\equiv 0$. Define $\partial D\in Q(k[\delta])$ as in Proposition \ref{prop:tangent}, and let $L^\sharp_{\partial D}$ be the induced automorphism of $k[\delta]\otimes_k A\cong A[\delta]/(\delta^2)$ from \eqref{eq:lx-sharp-full}. 
    
    Since $Q$ is involutory, Lemma \ref{lem:ld-full} shows that
    \[
    f= (L^\sharp_{\partial D}\circ L^\sharp_{\partial D})(f)=f+2\delta\psi_D(f)
    \]
    for all functions $f\in A$. It follows that $2\psi_D\equiv 0$; since $\operatorname{char}(k)\neq 2$, this completes the proof.
\end{proof}

\begin{ex}\label{ex:triv-la}
    Let $Q=\Spec(A)$ be a trivial affine quandle scheme (see Example \ref{ex:triv-1}); in particular, $Q$ is involutory. Unsurprisingly, $\q$ is abelian even when $\operatorname{char}(k)=2$. Indeed, for all $k$-derivations $D,E\in\q$ and functions $f\in A$, Example \ref{ex:triv-2} shows that
    \[
    [D,E](f)=(D\otimes E)(\nabla f)=(D\otimes E)(1\otimes f)=D(1) E(f)=0
    \]
    because $D(1)=0$.
\end{ex}

\subsubsection{}
Proposition \ref{prop:inv} and Example \ref{ex:triv-la} recover the following due to Example \ref{ex:conj} and Remark~\ref{rmk:recover}.

\begin{cor}\label{cor:inv}
    Let $G$ be an affine algebraic group. If $G$ is abelian, then $\operatorname{Lie}(G)$ is also abelian. If $\operatorname{char}(k)\neq 2$, then (more generally) $\operatorname{Lie}(G)$ is abelian if $g^2\in Z(G)$ for all $g\in G$.
\end{cor}

\begin{rmk}
    The second part of Corollary \ref{cor:inv} (and, hence, Proposition \ref{prop:inv}) fails if the assumption that $\operatorname{char}(k)\neq 2$ is dropped. For example, let $G$ be the Heisenberg group. If $\operatorname{char}(k)=2$, then $g^2\in Z(G)$ for all $g\in G$. However, the Lie algebra of $G$ is the $3$-dimensional Heisenberg Lie algebra, which is not abelian. 
    The converse of Corollary \ref{cor:inv} also does not hold; for example, see \cite{milne}*{10.35}.
\end{rmk}

Finally, combining Propositions \ref{prop:aff-racks}, \ref{prop:coass}, and \ref{prop:cocomm} with Example \ref{ex:triv-la} yields the following.
\begin{prop}
    Let $Q$ be an affine algebraic rack whose left-distributive operation $\tr$ is either associative, commutative, or equal to $\tr\inv$. Then the Leibniz algebra of $Q$ is abelian.
\end{prop}

\section{Yang--Baxter operators}\label{sec:ybo}

The next three sections of this paper are dedicated to studying the relationship between rack schemes $\Racks(\Sch
)$ and the Yang--Baxter equation. 
In this section, we recall the definition of Yang--Baxter operators in tensor categories introduced in \cite{joyal}; see \cite{kassel}*{p.\ 323} for a general reference and cf.\ \citelist{\cite{lebed}\cite{guccione}}.

\subsection{}
Let $(\C,\otimes)$ be a monoidal category. A \emph{Yang--Baxter operator} or \emph{braided object} in $\C$ is a pair $(X,r)$ where $X$ is an object in $\C$ and $r\in\operatorname{Aut}(X^2)$ is an automorphism that makes the following diagram commute:

\begin{equation}\label{dia:yb}
    \begin{tikzcd}
	{X^3} & {X^3} & {X^3} \\
	{X^3} & {X^3} & {X^3}
	\arrow["{r\otimes \id}", from=1-1, to=1-2]
	\arrow["{\id\otimes r}"', from=1-1, to=2-1]
	\arrow["{\id\otimes r}", from=1-2, to=1-3]
	\arrow["{r\otimes \id}", from=1-3, to=2-3]
	\arrow["{r\otimes \id}", from=2-1, to=2-2]
	\arrow["{\id\otimes r}", from=2-2, to=2-3]
\end{tikzcd}
\end{equation}

Moreover, if $(\C,\times)$ is cartesian monoidal and $(X,r)$ is a Yang--Baxter operator in $\C$, then let $\pi_1,\pi_2\colon X^2\to X$ denote the projection morphisms. Define endomorphisms $r_1,r_2\colon X^2\to X^2$ by
\begin{equation}\label{eq:r1}
    r_1\coloneq (\pi_1,\pi_1\circ r),\qquad r_2\coloneq (\pi_2\circ r,\pi_2),
\end{equation}
so that \[
r=(\pi_2\circ r_1,\pi_1\circ r_2).
\]
    We say that $(X,r)$ is \emph{nondegenerate} if $r_1$ and $r_2$ are invertible; cf.\ \cite{guccione}*{Prop.\ 3.3}. 
    
    If $(\C,\otimes)$ is cocartesian monoidal, then we dually define \emph{co-nondegenerate} Yang--Baxter operators.
    
For example, if $\C$ is a braided monoidal category and $X$ is an object in $\C$, consider the braiding $\tau\colon X^2\bij X^2$. Then the pair $(X,\tau)$ is a Yang--Baxter operator; if $\C$ is cartesian (resp.\ cocartesian) monoidal, then $(X,\tau)$ is nondegenerate (resp.\ co-nondegenerate).

\subsection{}
A \emph{morphism of Yang--Baxter operators} $(X,r)\to (Y,s)$ is a morphism $\phi\colon X\to Y$ in $\C$ that makes the following diagram commute:
\[\begin{tikzcd}
	{X^2} & {X^2} \\
	{Y^2} & {Y^2}
	\arrow["r", from=1-1, to=1-2]
	\arrow["{\phi\otimes\phi}"', from=1-1, to=2-1]
	\arrow["{\phi\otimes\phi}", from=1-2, to=2-2]
	\arrow["s", from=2-1, to=2-2]
\end{tikzcd}\]

Thus, Yang--Baxter operators in $\C$ form a category $\YB(\C)$. If $\C$ is cartesian monoidal (resp.\ cocartesian monoidal), then let $\nd(\C)$ be the full subcategory of $\YB(\C)$ whose objects are nondegenerate (resp.\ co-nondegenerate). 

The following is straightforward.
\begin{prop}\label{prop:yb-op}
    Let $\C$ and $\D$ be symmetric monoidal categories, and let $F\colon \C\to\D$ be a strong monoidal equivalence (resp.\ anti-equivalence) of categories. Then $F$ induces an equivalence (resp.\ anti-equivalence) of categories $\YB(\C)\bij\YB(\D)$. In particular, $\YB(\C)\cong\YB(\C\op)$. 
    
    Moreover, if $\C$ is (co)cartesian monoidal, then $F$ induces an equivalence (resp.\ anti-equivalence) of categories $\nd(\C)\bij\nd(\D)$. 
\end{prop}

\subsection{}
For example, if $\C=(\Set,\times)$, then $\YB(\C)$ is the category of \emph{braided sets} or \emph{set-theoretic solutions} to the Yang--Baxter equation. This category was introduced by Drinfeld \cite{drinfeld} in 1992 and has enjoyed significant study since then, often in relation to racks. See, for example, \citelist{\cite{etingof}\cite{etingof2}\cite{grana}}.

Given a braided set $(X,r)$, authors often denote 
\[
r(x,y)\eqcolon(\sigma_x(y),\tau_y(x)).
\]
Then $(X,r)$ is nondegenerate in the above sense if and only if for all $x\in X$, the maps $\sigma_x,\tau_x\colon X\to X$ are permutations. The latter condition is the usual definition of nondegeneracy for braided sets.

\begin{com}\label{com:yb-motivations}
We discuss several motivations for studying Yang--Baxter operators at this level of generality. 
    For symmetric monoidal categories $\C$, composition with $\tau$ provides a well-known one-to-one correspondence between Yang--Baxter operators $\YB(\C)$ and solutions to the \emph{quantum Yang--Baxter equation} in $\C$; cf.\ \cite{guccione}*{Prop.\ 3.9}. 
    
In particular, solutions in the category of vector spaces $\C=(\mathsf{Vect}_k,\otimes)$, which are often called \emph{$R$-matrices}, were introduced by Yang \cite{yang} in 1967 and Baxter \cite{baxter} in 1972 in the context of statistical mechanics and quantum many-body theory; see \cite{kassel}*{Ch.\ VIII} for an introduction to $R$-matrices. 
Aside from their importance in physics and the theory of quantum groups, $R$-matrices are the original motivation for braided sets. Indeed, the free vector space functor sends every braided set to an $R$-matrix; see, for example, \citelist{\cite{drinfeld}\cite{etingof}}. 

Various authors have studied Yang--Baxter operators in more complicated symmetric monoidal categories. These include, for example, the category of irreducible complex algebraic varieties and birational morphisms \cite{etingof2} and $(\mathsf{Cat},\times)$, the category of small categories and functors \cite{jiang}. These examples motivate the general categorical framework discussed in this section and the sequel.
\end{com}

\section{From rack objects to Yang--Baxter operators}\label{sec:from2}
In this section, let $(\C,\times)$ be a cartesian monoidal category.
In 2004, Crans \cite{crans}*{Sec.\ 3.1.4} used rack objects to construct Yang--Baxter operators in $(\C,\times)$. This work vastly generalized a 1988 construction of Brieskorn \cite{brieskorn} (see \cite{grana}*{Sec.\ 5.1}) in $\Set$, and it was itself generalized to the category of cocommutative coalgebras in a symmetric monoidal category in \citelist{\cite{carter}\cite{lebed}\cite{guccione}}. In this section, we show that the construction of Crans induces a fully faithful functor $F\colon \Racks(\C)\to \YB^{\mathrm{nd}}(\C)$.

\subsection{}
Recall that for all objects $X$ in $\C$, the braiding $X^2\bij X^2$ is denoted by $\tau=(\pi_2,\pi_1)$. Attempt to define $F\colon \Racks(\C)\to \YB^{\mathrm{nd}}(\C)$ on objects by assigning each rack object $(X,\tr,\tr\inv)$ to the pair $(X,r)$, where $r\colon X^2\to X^2$ is the morphism
\[
r\coloneq (\pi_1,\tr)\circ\tau.
\]
(This is a more concise formula for the braiding $B$ in \cite{crans}*{Lem.\ 81}.) 
Also, let $F$ fix all morphisms. 

The remainder of this section is dedicated to proving the following. 

\begin{thm}\label{thm:yb}
    The assignment $F\colon\Racks(\C)\to\YB^{\mathrm{nd}}(\C)$ is a fully faithful functor.
\end{thm}

\begin{proof}
    We prove the theorem piecemeal in Propositions \ref{prop:yb1}--\ref{prop:yb2}.
\end{proof}

\subsection{}
First, we review the action of $F$ on objects. As discussed in \cite{guccione}*{Prop.\ 4.4}, the following two results can be shown for cocommutative coalgebras in symmetric monoidal categories using string diagrams. For the reader's convenience, we give independent symbolic proofs in the cartesian monoidal case.

In the following, let $(X,\tr,\tr\inv)$ be a rack object in $\C$, and let $(X,r)$ denote its image under $F$. 

\begin{prop}[\cite{guccione}*{Prop.\ 4.4}]\label{prop:yb1}
    $(X,r)$ is a Yang--Baxter operator.
\end{prop}

\begin{proof}
    First, we have to show that $r\colon X^2\to X^2$ is invertible. To that end, define a morphism $r\inv\colon X^2\to X^2$ by
    \[
    r\inv\coloneq (\tr\inv,\pi_2)\circ\tau.
    \]
    The reader can verify from the invertability axiom \eqref{ax:inv} that $r$ and $r\inv$ are mutually inverse.

    It remains to show that the diagram \eqref{dia:yb} commutes. Define morphisms $\phi,\psi\colon X^3\to X^3$ by
    \[
    \phi\coloneq (r\times\id)(\id\times r)(r\times\id),\qquad \psi\coloneq (\id\times r)(r\times\id)(\id\times r),
    \]
    and let $p_1,p_2,p_3\colon X^3\to X$ be the projection morphisms. 
    
    It suffices to show that $p_i\circ \phi=p_i\circ\psi$ for all $i=1,2,3$. 
    Indeed, the reader can verify that
    \[
    p_1\circ\phi=p_3=p_1\circ\psi,\qquad p_2\circ\phi=\tr\circ(p_3,p_2)=p_2\circ\psi.
    \]
    
    Moreover, consider the morphism $\sigma\coloneq(p_3,p_2,p_1)$. The reader can verify that
    \[
    p_3\circ\phi=\tr\circ(p_3,\tr\circ(p_2,p_1))=\tr\circ(\id\times\tr)\circ\sigma
    \]
    and
    \[
    p_3\circ\psi=\tr\circ(\tr\circ(p_3,p_2),\tr\circ(p_3,p_1))=\tr\circ(\id\times\tr)\circ((\tr,\pi_1)\times\id)\circ\sigma.
    \]
    By the left-distributivity axiom \eqref{ax:ld}, these two expressions are equal.
\end{proof}

\begin{prop}[\cite{guccione}*{Prop.\ 4.4}]
    $(X,r)$ is nondegenerate.
\end{prop}

\begin{proof}
    Define $r_1,r_2\colon X^2\to X^2$ as in \eqref{eq:r1}. Then
    \[
    r_1=(\pi_1,\pi_1\circ(\pi_1,\tr)\circ\tau)=(\pi_1,\pi_1\circ\tau)=(\pi_1,\pi_2)=\id_{X^2},
    \]
    which is invertible. Furthermore, since $r$ is invertible (see Proposition \ref{prop:yb1}), the composite
    \[
    r_2=(\pi_2\circ(\pi_1,\tr)\circ\tau,\pi_2)=(\tr\circ\tau,\pi_2)=(\tr,\pi_1)\circ\tau=\tau\circ(\pi_1,\tr)\circ\tau=\tau\circ r
    \]
    is also invertible.
\end{proof}

\subsection{}
Next, we study the action of $F$ on morphisms. The following two results are known in the special case that $\C=\Set$; see the end of \cite{grana}*{Sec.\ 5.1}. However, the general forms of these results appear to be new.

In the following, let $(X,\tr_X,\tr_X\inv)$ and $(Y,\tr_Y,\tr_Y\inv)$ be rack objects in $\C$, and let $(X,r)$ and $(Y,s)$ denote their respective images under $F$. 

\begin{prop}\label{prop:f-mor}
    $F$ sends morphisms of rack objects to morphisms of Yang--Baxter operators.
\end{prop}

\begin{proof}
    Let $\phi\colon X\to Y$ be a morphism in $\Racks(\C)$. We have to show that $\phi$ is also a morphism in $\YB(\C)$. By the definition of $\phi\times\phi=(\phi\circ\pi_1,\phi\circ\pi_2)$, we have $\phi\circ\pi_i=\pi_i\circ(\phi\times\phi)$ for $i=1,2$. Thus,
    \begin{align*}
        (\phi\times\phi)\circ r &= (\phi\circ\pi_1,\phi\circ\tr_X)\circ\tau\\
        &= (\pi_1\circ (\phi\times \phi),\tr_Y\circ(\phi\times\phi))(\pi_2,\pi_1)\\
        &= (\pi_1,\tr_Y)(\pi_2,\pi_1)(\phi\times\phi)\\
        &= (\pi_1,\tr_Y)\circ\tau\circ(\phi\times\phi)\\
        &= s\circ(\phi\times\phi),
    \end{align*}
    where in the second equality we have used the fact that $\phi$ is a morphism of rack objects.
\end{proof}

\begin{prop}\label{prop:yb2}
    $F$ is a fully faithful functor.
\end{prop}

\begin{proof}
    Thanks to Proposition \ref{prop:f-mor}, the only nontrivial part of the claim is fullness. Let $\phi\colon (X,r)\to(Y,s)$ be a morphism in $\YB(\C)$. We have to show that $\phi$ is also a morphism in $\Racks(\C)$. Note that
    \[
    \tr_X\circ\tau=\pi_2\circ(\pi_1,\tr_X)\circ\tau=\pi_2\circ r\colon X^2\to X
    \]
    and, similarly, $\tr_Y\circ\tau=\pi_2\circ s\colon Y^2\to Y$. Hence,
    \begin{align*}
        \phi\circ \tr_X\circ\tau&=\phi\circ\pi_2\circ r\\
        &=\pi_2\circ(\phi\times\phi)\circ r\\
        &=\pi_2\circ s\circ(\phi\times\phi)\\
        &=\tr_Y\circ\tau\circ(\phi\times\phi)\\
        &= \tr_Y\circ(\phi\times\phi)\circ\tau,
    \end{align*}
    where in the third equality we have used the fact that $\phi$ is a morphism of Yang--Baxter operators. Since $\tau\colon X^2\bij X^2$ is invertible, the claim follows.
\end{proof}

This completes the proof of Theorem \ref{thm:yb}.

\section{From racks to Yang--Baxter schemes}\label{sec:from3}

In this section, we discuss several applications of Sections \ref{sec:ybo} and \ref{sec:from2} to the category of schemes $\Schs$ over a commutative ring $S$.

\begin{definition}
    The category of \emph{Yang--Baxter schemes} over $S$ is $\YB(\Schs)$. \emph{Affine Yang--Baxter schemes} and \emph{Yang--Baxter varieties} are defined similarly.
\end{definition}

\subsection{Discussion}
Yang--Baxter schemes are of interest because they provide numerous solutions to the Yang--Baxter equations in certain categories, as we explain below. 
In particular, Yang--Baxter varieties with birational morphisms were previously studied in, for example, \citelist{\cite{etingof}\cite{adler}\cite{suris}\cite{frieden}}.

Via the functor of points approach, every (affine) Yang--Baxter scheme over $S$ can be viewed as a (representable) functor $\Algs\to \Set$ that factors through $\YB(\Set)$, the category of braided sets. In this view, nondegenerate Yang--Baxter schemes are precisely those that factor through $\nd(\Set)$, the category of nondegenerate braided sets. 
Therefore, for every Yang--Baxter scheme $(X,r)$ and every $S$-algebra $R$, the set of $R$-points is a braided set $(X(R),r_R)$. Moreover, if $(X,r)$ is nondegenerate, then so is $(X(R),r_R)$. Simply put, Yang--Baxter schemes generate braided sets from $S$-algebras.

On the other hand, if $(X,r)$ is a Yang--Baxter scheme over $S$, then let $U$ be an affine open subscheme of $X$ such that $r(U^2)\subseteq U^2$. Since $r$ restricts to an automorphism of $U^2$, the pair $(U,r|_{U^2})$ is an affine Yang--Baxter scheme. It follows from Proposition \ref{prop:yb-op} that $(\O_X(U),(r|_{U^2})^\sharp)$ is a Yang--Baxter operator in $(\Algs,\otimes_S)$, and the former is nondegenerate if and only if the latter is co-nondegenerate. Moreover, if $S=k$ is a field, then we can view the latter as an object in $\YB(\mathsf{Vect}_k,\otimes)$, that is, an $R$-matrix in the sense of \cite{kassel}*{Ch.\ VIII}. In this sense, Yang--Baxter schemes turn affine open subschemes preserved by $r$ into $R$-matrices with $S$-algebra structures.

Several months after this paper was posted to arXiv, Ma, Zhang, and Liu \cite{braces}*{Sec.\ 6} independently reintroduced affine Yang--Baxter schemes; their definition is equivalent to ours by the above discussion. Similarly to the results below, the authors showed that skew brace objects and Rota--Baxter group objects in $\Aff$ give rise to objects in $\YB(\Aff)$.

\subsection{Solutions from rack schemes}

Theorem \ref{thm:yb} states that every rack scheme $(Q,\tr)$ induces a nondegenerate Yang--Baxter scheme $(Q,r)$; this assignment defines a fully faithful functor $\Racks(\Schs)\to\nd(\Schs)$, and similarly for $\Affs$. Namely, for all $S$-algebras $R$, the rack $(Q(R),\tr_{Q(R)})$ becomes a nondegenerate braided set $(Q(R),r_R)$ via the formula
\begin{equation}\label{eq:r-rack}
    r_R(x,y)\coloneq(y,y\tr_{Q(R)} x).
\end{equation}

\subsubsection{}
By Proposition \ref{prop:yb-op}, the usual anti-equivalence of categories $\Affs\to \Algs$ descends to an anti-equivalence of categories
\[
\nd(\Affs,\times_S)\to\nd(\Algs,\otimes_S).
\]
Its composite with the functor $F$ from Theorem \ref{thm:yb} is a fully faithful contravariant functor
\[
\Racks(\Affs)\to\nd(\Algs,\otimes_S)\]
that sends affine rack schemes $(\Spec(A), \tr)$ to co-nondegenerate Yang--Baxter operators $(A,r^\sharp)$ in $\Algs$. 
By \eqref{eq:r-rack}, the structure morphism $r^\sharp\in\operatorname{Aut}(A\otimes_S A)$ is defined on elementary tensors by
\begin{equation}\label{eq:r-sharp-rack}
    r^\sharp(f\otimes g)=\tau((f\otimes 1)\nabla g)=\sum\gt\otimes f\go,
\end{equation}
where $\nabla g=\sum\go\otimes\gt$.

\subsubsection{}
To illustrate the constructions in \eqref{eq:r-rack} and \eqref{eq:r-sharp-rack}, we compute the Yang--Baxter operators induced by the rack schemes in Examples \ref{ex:perm} and \ref{ex:conj}. See also Section \ref{subsec:yb-ol}.
\begin{ex}
    If $Q$ is a permutation rack scheme over $S$ with respect to $\phi\in\operatorname{Aut}(Q)$ as in Example \ref{ex:perm}, then the corresponding nondegenerate Yang--Baxter scheme $(Q,r)$ sends every $S$-algebra $R$ to the nondegenerate braided set $(Q(R),r_R)$, where 
    \[
    r_R(x,y)=(y,\phi_R(x)).
    \]
    In particular, if $Q$ is a trivial quandle scheme, then $r=\tau$.

    Moreover, if $Q=\Spec(A)$ is affine, then \eqref{eq:r-sharp-rack} and Remark \ref{rmk:perm-nabla} show that the induced co-nondegenerate Yang--Baxter operator $(A,r^\sharp)$ in $\Algs$ is given by
    \[
    r^\sharp(f\otimes g)=\phi^\sharp(g)\otimes f.
    \]
\end{ex}

\begin{ex}\label{ex:yb-conj}
    Let $\Conj(G)$ be the conjugation quandle scheme of an algebraic group $G$ over a field $k$. The corresponding nondegenerate Yang--Baxter scheme $(G,r)$ sends every $k$-algebra $R$ to the nondegenerate braided set $(G(R),r_R)$, where 
    \[
    r_R(g,h)=(h,hgh\inv).
    \]

    If $G=\Spec(A)$ is affine, then $G$ induces a canonical Hopf algebra structure $(\Delta,\eps,S)$ on $A$. Together, \eqref{eq:r-sharp-rack} and Example \ref{ex:corack-conj} show that the induced co-nondegenerate Yang--Baxter operator $(A,r^\sharp)$ in $\Alg$ is given by
    \[
    r^\sharp(f\otimes g)=\sum \gt\otimes f\go S(g_{(3)})
    \]
    for all functions $f,g\in A$, where the right-hand side uses Sweedler notation for $\Delta$ (rather than $\nabla$).
\end{ex}

\begin{rmk}
    Since the category of affine algebraic groups $\Grp(\Aff)$ is anti-equivalent to the category of commutative Hopf algebras over $k$, Example \ref{ex:yb-conj} assigns a canonical co-nondegenerate Yang--Baxter operator $(A,r^\sharp)$ to every commutative Hopf algebra $A$. 
    
    It is interesting to compare $(A,r^\sharp)$ with the Yang--Baxter operator $(A,T)$ that Woronowicz \cite{woro} (cf.\ \cite{carter}*{Prop 4.7}) constructed in 1991. Here, $T(f\otimes g)= \sum\gt\otimes fS(\go)g_{(3)}$. (Note that $A$ need not be commutative in Woronowicz's construction.) These two Yang--Baxter operators are similar but not necessarily isomorphic, as the reader can verify in the case that $A=\O(\GL_2)$.
\end{rmk}

\begin{rmk}
    In Example \ref{ex:yb-conj}, it is natural to ask whether $(G,r)$ is a Yang--Baxter operator in the category of algebraic groups $\Grp(\Sch)$. This is equivalent to the condition that $r\colon G^2\to G^2$ is a morphism of algebraic groups, which holds if and only if $G$ is abelian. In that case, $\Conj(G)$ is a trivial quandle scheme and, accordingly, $r=\tau$. Interestingly, a similar result holds for the Yang--Baxter operator $(A,T)$ constructed in \cite{woro}.
\end{rmk}

\subsection{A note on Leibniz algebras}
Given an affine algebraic rack $(\Spec(A),\tr)$ over a field $k$ with left Leibniz algebra $\q=\Der_k(A,k)$, consider the morphism $r$ from \eqref{eq:r-rack} and its tangent map $r'\coloneq (dr)_{(e,e)}\colon \q^2\to\q^2$. 

Since $(\Spec(A),r)$ is a nondegenerate Yang--Baxter scheme, it is natural to ask whether the pair $(\q,r')$ is a nondegenerate Yang--Baxter operator in $(\leib,\oplus)$. In the remainder of this section, we show that this question has a positive answer, albeit an uninteresting one. (This does not follow from Theorem \ref{thm:main} because $r$ is not necessarily a morphism of algebraic racks.)

\subsubsection{}
In the following, identify $\q^2\bij\Der_k(A\otimes_k A,k)$
via the isomorphism
\[
(D,E)\mapsto \omega_{D,E},\qquad \omega_{D,E}(f\otimes g)\coloneq D(f)\eps(g)+\eps(f)E(g).
\]

\begin{rmk}
    The inverse of the above isomorphism is
    \[
    \omega\mapsto (D_\omega,E_\omega),\quad D_\omega(f)\coloneq\omega(f\otimes 1),\quad E_\omega\coloneq \omega(1\otimes f).
    \]
\end{rmk}

\begin{prop}
    If $\Spec(A)$ is an affine algebraic rack over $k$ with Leibniz algebra $\q$, then $r'=\tau$.
\end{prop}

\begin{proof}
    We have to show that $r'(\omega_{D,E})=\omega_{E,D}$ for all $k$-derivations $D,E\in\q$. 
    To that end, let $f,g\in A$, and write $\nabla g=\sum\go\otimes\gt$. Using the definition of $r'$, the formula \eqref{eq:r-sharp-rack}, and the $k$-linearity of $D$ and $E$, we compute
    \begin{align*}
        r'(\omega_{D,E})(f\otimes g)&=\omega_{D,E}(r^\sharp(f\otimes g))\\
        &= \omega_{D,E}\left(\sum\gt\otimes f\go\right)\\
        &=\sum D(\gt)\eps(f\go)+\sum \eps(\gt)E(f\go)\\
        &= \eps(f)\sum D(\eps(\go)\gt) +\sum\eps(\gt)E(f)\eps(\go)+\sum\eps(\gt)\eps(f)E(\go)\\
        &=\eps(f)D\left(\sum\eps(\go)\gt\right)+E(f)\eps\left(\sum\eps(\go)\gt\right)+\eps(f)E\left(\sum \go\eps(\gt)\right) \\
        &=\eps(f)D(g)+E(f)\eps(g)+\eps(f) E(\eps(g))\\
        &=\omega_{E,D}(f\otimes g)+0.
    \end{align*}
    The fourth and fifth equalities use the fact that $\eps\colon A\to k$ is a $k$-algebra homomorphism. The fourth and seventh equalities use the fact that $E$ is a $k$-derivation. The sixth equality follows after applying the co-fixing identity \ref{ax:co-fixing} to the first two summations and the co-fixedness identity \ref{ax:co-fixed} to the third summation. 
    Since $f,g\in A$ were arbitrary, the proof is complete.
\end{proof}

\section{Worked example}\label{sec:ol}
To illustrate the constructions in Theorem \ref{thm:main}, we compute the left Leibniz algebra of the omni-linear rack variety $\OL_n$ over a field $k$. This provides a class of affine algebraic racks whose Leibniz algebras are not Lie algebras. To illustrate the constructions in Section \ref{sec:from3}, we also compute the objects in $\nd(\Aff)$ and $\nd(\Alg)$ induced by $\OL_n$.

\subsection{Leibniz algebra computation}
We give an algebraic analogue of a calculation of Kinyon \cite{kinyon} from the theory of Lie racks. Recall that, as an affine variety, the omni-linear rack variety is
\[
\OL_n=\GL_n\times\mathbb{A}^n.
\]
Therefore, as a vector space over $k$, the tangent space of $\OL_n$ at the identity $e=(I_n,\mathbf{0})$ is
\[
T_e(\OL_n)=\mathfrak{gl}_n(k)\oplus k^n.
\]

\subsubsection{}
Theorem \ref{thm:main} equips the vector space of $k$-derivations $\q\coloneq \Der_k(\O(\OL_n),k)$ with a canonical left Leibniz bracket $[\cdot,\cdot]$. We show that under the identification $\q\cong T_e(\OL_n)$, this bracket agrees with that of the omni-Lie algebra $\ol_n(k)$ from \eqref{eq:olnk}.

Recall from \eqref{eq:ol-corack} that the subset $\{s_{11},s_{12},\dots,s_{nn},{\det}\inv,t_1,\dots,t_n\}$ generates $\O(\OL_n)$ as a $k$-algebra.   
When viewed as an element of $\q$, each vector $(X,v)\in T_e(\OL_n)$ acts on $\O(\OL_n)$ by
\[
(X,v)(s_{ij})= x_{ij},\quad (X,v)({\det}\inv)=-\operatorname{tr}(X),\quad (X,v)(t_k)=v_k.
\]
\begin{prop}\label{prop:leib-of-ol}
    For all $n\geq 0$, the left Leibniz algebra of the omni-linear rack variety $\OL_n$ is the omni-Lie algebra $\ol_n(k)$.
\end{prop}

\begin{proof}
By Remark \ref{rmk:recover}, the formula
\[
(A,v)\tr(B,w)=(ABA\inv,Aw)
\]
shows that the Leibniz bracket of $T_e(\OL_n)$ agrees with the commutator on $\mathfrak{gl}_n(k)\leq T_e(\OL_n)$. 

We compute the value of the Leibniz bracket on the second direct summand $k^n\leq T_e(\OL_n)$. First apply the convolution formula in Theorem \ref{thm:main} and then apply \eqref{eq:ol-nabla} to obtain
\begin{align*}
    [(X,v),(Y,w)](t_k)&=((X,v)\otimes (Y,w))(\nabla t_k)\\
    &=((X,v)\otimes (Y,w))\left(\sum^n_{i=1}s_{ki}\otimes t_i\right)\\
    &=\sum^n_{i=1}x_{ki}w_i = (Xw)_k
\end{align*}
for all $(X,v),(Y,w)\in T_e(\OL_n)$ and $1\leq k\leq n$. Hence, the Leibniz bracket of $T_e(\OL_n)$ is
\[
[(X,v),(Y,w)]=(XY-YX,Xw),
\]
which is precisely the Leibniz bracket of $\ol_n(k)$ from \eqref{eq:olnk}.
\end{proof}

\begin{rmk}
    As predicted by the last part of Theorem \ref{thm:main}, Proposition \ref{prop:leib-of-ol} recovers \cite{kinyon}*{(9)} in the case that $H=\GL_n(k)$ and $V=k^n$ when $k=\R$ or $\mathbb{C}$.
\end{rmk}

\begin{com}\label{com:rep}
    In 2024, Liu and Sheng \cite{omni} introduced \emph{omni-representations} of Leibniz algebras, which are homomorphisms into the omni-Lie algebra $\ol(V)$. In particular, every left Leibniz algebra $\q$ embeds into $\ol(\q)$ via the \emph{adjoint omni-representation} $X\mapsto ([X,\cdot],X)$; see \cite{kinyon2}*{Cor.\ 3.2} or \cite{omni}*{Prop.\ 3.1} for further discussion. Since $\ol(V)$ has the same underlying vector space as $\mathfrak{aff}(V)$ (the Lie algebra of the group of affine transformations $\mathrm{Aff}(V)$), it follows that every finite-dimensional Leibniz algebra can be realized as a matrix Leibniz algebra.

    Similarly, one can define an \emph{omni-representation} of an affine algebraic rack to be a morphism into the omni-linear rack variety $\OL_n$, which may be identified with an action on affine space $\mathbb{A}^n$ by affine transformations. Since the construction in Theorem \ref{thm:main} is functorial, Proposition \ref{prop:leib-of-ol} implies that omni-representations of affine algebraic racks induce omni-representations of their Leibniz algebras. This may allow for the development of a representation theory of affine algebraic racks that simultaneously encompasses those of affine algebraic groups (see \citelist{\cite{milne}\cite{waterhouse}} for references) and finite racks (as recently developed in \citelist{\cite{nieto}\cite{elhamdadi2}}).
\end{com}

\subsection{Yang--Baxter computations}\label{subsec:yb-ol}

By \eqref{eq:r-rack}, the omni-linear rack variety $\OL_n$ induces a nondegenerate affine Yang--Baxter variety $(\OL_n,r)$. Namely, for all $k$-algebras $R$, the induced nondegenerate braided set $(\OL_n(R),r_R)$ is given by
\[
r_R((A,v),(B,w))=((B,w),(BAB\inv,Bv)).
\]

Hence, in the notation of \eqref{eq:ol-corack}, the induced co-nondegenerate Yang--Baxter operator $(\O(\OL_n),r^\sharp)$ in $\Alg$ from \eqref{eq:r-sharp-rack} is given by
\[
r^\sharp((s_{ab}\otimes t_p)\otimes (s_{cd}\otimes t_q))=\sum_{1\leq i,j,\ell\leq n}(s_{ij}\otimes t_\ell)\otimes(s_{ab}s_{ci}s\inv_{jd}s_{q\ell}\otimes t_p),
\]
where $s\inv_{jd}\coloneq S(s_{jd})\in\O(\GL_n)$ sends $B\in\GL_n(k)$ to the $(j,d)$-entry of $B\inv$.
This is because 
\begin{align*}
    r^\sharp((s_{ab}\otimes 1)\otimes(1\otimes 1))&=(1\otimes 1)\otimes (s_{ab}\otimes 1),\\
    r^\sharp((1\otimes t_p)\otimes(1\otimes 1))&=(1\otimes 1)\otimes (1\otimes t_p),\\
    r^\sharp((1\otimes 1)\otimes(s_{cd}\otimes 1))&=\sum_{1\leq i,j\leq n}(s_{ij}\otimes 1)\otimes (s_{ci}s_{jd}\inv\otimes 1),\\
    r^\sharp((1\otimes 1)\otimes(1\otimes t_q))&=\sum^n_{\ell=1}(1\otimes t_\ell)\otimes (s_{q\ell}\otimes 1),
\end{align*}
where in the last two equations we have used the identities
\[
(BAB\inv)_{cd}=\sum_{1\leq i,j\leq n}B_{ci}A_{ij}(B\inv)_{jd},\qquad (Bv)_q=\sum^n_{\ell=1}B_{q\ell}v_\ell.
\]

\section{Open questions}\label{sec:open}
We conclude by proposing directions for further research. Let $\q$ be a Leibniz algebra.

\begin{enumerate}
    \item If $Q$ is \emph{any} algebraic rack (not necessarily affine), then does $T_eQ$ have canonical Leibniz brackets? That is, do the functors $\Rack(\Aff)\to\leib$ and $\Rack(\Aff)\to\Rleib$ from Theorem \ref{thm:main} extend to functors $\Rack(\Sch)\to\leib$ and $\Rack(\Sch)\to\Rleib$?
    \item If $G$ is an algebraic group, then its Lie algebra $\mathrm{Lie}(G)$ is isomorphic to the space of left-invariant vector fields, that is, the space of all $k$-derivations $D\colon \O(G)\to\O(G)$ that satisfy the identity $\Delta\circ D=(\mathord{\id}\otimes D)\circ\Delta$. (See, for example, \cite{milne}*{Sec.\ 10e}.) Do Leibniz algebras of affine algebraic racks have a similar interpretation?
    \item Mirroring the theory of representations of algebraic groups and comodules over commutative Hopf algebras (see, for example, \cite{milne}*{Sec.\ 4e}), develop a theory of affine representations of affine algebraic racks and ``comodules'' over corack algebras. Also, study their relationships with the induced omni-representations (see \cite{omni}) of Leibniz algebras. See Comment \ref{com:rep} for further discussion and references.
    \item Building upon the previous question, develop a cohomology theory for algebraic racks and their representations. Study its relationships with Leibniz algebra cohomology (see \cite{leibniz}), rack and quandle cohomology (see \cite{book}), and (for conjugation quandle schemes) the cohomology of algebraic groups and their representations (see \cite{milne}).
    \item More broadly, which theorems about affine algebraic groups and their Lie algebras extend to affine algebraic racks and their Leibniz algebras? (Cf.\ Remark \ref{rmk:triv}.)
    \item Determine which Leibniz algebras are \emph{algebraic} or ``integrate to an algebraic rack.'' That is, given $\q$, determine whether there exists an algebraic rack whose Leibniz algebra is $\q$.
    \item Further develop the theory of rack schemes $\Racks(\Schs)$ (as opposed to algebraic racks $\Rack(\Schs)$). In particular, generalize the results of Takahashi \citelist{\cite{takahashi}\cite{takahashi2}} from quandle varieties to rack schemes.
    \item Develop the general theory of corack algebras, perhaps in analogy to the theory of Hopf algebras. What relationships do they have with rack bialgebras (see \citelist{\cite{bardakov}\cite{alexandre}})?
    \item Racks \cite{fenn} and quandles \citelist{\cite{joyce}\cite{matveev}} were originally introduced to construct invariants of knots and 3-manifolds; see \citelist{\cite{book}\cite{quandlebook}}. In the past decade, racks (for example, \citelist{\cite{randal}\cite{shusterman}\cite{ellenberg}\cite{landesman}}) and quandles (for example, \citelist{\cite{fan}\cite{bianchi}\cite{bianchi2}\cite{takahashi1}}) have also enjoyed applications as invariants of various algebro-geometric objects. Use rack schemes to enhance these invariants, and compare them to other algebro-geometric invariants.
\end{enumerate}

\bibliographystyle{amsplain}

\end{document}